\documentclass[10pt]{article}
\usepackage{amsmath, amssymb ,amsthm, amsfonts, amsgen, color}
\usepackage{graphicx}

\numberwithin{equation}{section} 
\newcommand{\Nb}{{\mathbb{N}}}
\newcommand{\Rb}{{\mathbb{R}}}

\newcommand{\wto}{\rightharpoonup}
\newcommand{\wsto}{\stackrel{*}{\rightharpoonup}}

\makeatletter
\def\rightharpoonupfill@{\arrowfill@\relbar\relbar\rightharpoonup}
\newcommand{\xrightharpoonup}[2][]{\ext@arrow
0359\rightharpoonupfill@{#1}{#2}} \makeatother

\def\weakstar{\buildrel\star\over\rightharpoonup}
\def\e{{\varepsilon}}
\def\O{{\Omega}}

\def\weak{\rightharpoonup}

\newtheorem{theorem}{Theorem}[section]
\newtheorem{lemma}[theorem]{Lemma}
\newtheorem{proposition}[theorem]{Proposition}

\newtheorem{remark}[theorem]{Remark}
\newtheorem{definition}[theorem]{Definition}

\newcommand{{\rr}}{{\mathbb R}}

\newenvironment{@abssec}[1]{%
     \if@twocolumn
       \section*{#1}%
     \else
       \vspace{.05in}\footnotesize
       \parindent .2in
         {\upshape\bfseries #1. }\ignorespaces
     \fi}
     {\if@twocolumn\else\par\vspace{.1in}\fi}

\begin{document}

\title{\sc Relaxation for optimal design problems with non-standard growth }

\author{{\sc Ana Cristina Barroso}\\
Departamento de Matem\'atica and CMAF-CIO \\ 
Faculdade de Ci\^encias da Universidade de Lisboa\\ 
Campo Grande, Edif\' \i cio C6, Piso 1\\
1749-016 Lisboa, Portugal\\
acbarroso@ciencias.ulisboa.pt\\
and \\
{\sc Elvira Zappale} \\
	Dipartimento di Ingegneria Industriale \\
	Universit\`a degli Studi di Salerno\\
	Via Giovanni Paolo II, 132\\
	84084 Fisciano (SA), Italy\\
	and CIMA (Universidade de \'Evora)\\
ezappale@unisa.it}
\maketitle

This paper is dedicated to the memory of our good friend and 
colleague Gra\c ca Carita, with whom the second author had the privilege of collaborating on many previous occasions.  
Gra\c ca joined us at the early stages of 
this work and tragically passed away on September 26, 2016.

\begin{abstract}
In this paper we investigate the possibility of obtaining a measure 
representation for functionals arising in the context of optimal 
design problems under non-standard growth conditions and perimeter penalization. Applications to modelling of strings are also provided.

\end{abstract}

\smallskip



\section{Introduction}

Given two conductive materials present in a container $\Omega$, 
and prescribing
the volume fraction of each one, the optimal design problem, 
originally studied by Murat and Tartar and by Kohn and Strang in
\cite{MT, KS1, KS2, KS3}, consists of identifying the minimal energy 
configuration of the mixture.
As emphasized in \cite{MT} (see, for instance, the one dimensional 
model case in Proposition 4 and Remark 7 therein), this problem 
might not have a solution. In the works of Ambrosio and Buttazzo 
\cite{AB} and Kohn and Lin \cite{KL} this difficulty is 
overcome by the introduction, in the energy functional to be 
minimized, of a term which penalizes the 
perimeter of the sets where the mixture equals one of the conductive 
materials, thus also 
eliminating the case where the two materials are finely mixed.

Carita and Zappale in \cite{CZ} considered a similar functional to 
the one in \cite{AB} and studied the minimum problem
\begin{eqnarray}
& & \inf \left \{ \int_{\Omega}\chi_E(x) W_1(\nabla u(x)) + 
(1 - \chi_E(x)) 
W_2(\nabla u(x)) \, dx 
+ |D\chi_E|(\Omega) : \right. \nonumber \\ 
& & \hspace{4cm} \left. u \in W^{1,1}(\Omega; \Rb^d), \chi_E \in 
BV(\Omega;\{0,1\}), u = u_0 \; {\rm on} \; \partial \Omega \right\},
\nonumber
\end{eqnarray}
where the densities $W_i$, $i = 1,2$, are continuous functions such 
that there exist positive constants $\alpha$ and $\beta$
for which
$$\alpha |\xi| \leq W_i(\xi) \leq \beta (1 + |\xi|), \; \; \forall 
\xi \in \Rb^{d \times N}.$$
Since no convexity assumptions were placed on $W_i$, they considered 
the relaxed localized energy arising
from  the above problem
\begin{align}
\mathcal{F}_{\mathcal{OD}}\left(  \chi,u;A\right)   &  
:=\inf\left\{  
\underset{n\rightarrow +\infty}{\lim\inf}\left[
 \int_{A}\chi_n(x) W_1(\nabla u_n(x)) + (1 - \chi_n(x))W_2(\nabla u_n(x)) \, dx + |D\chi_n|(A) \right]: \right.\nonumber \\
& \hspace{3cm}  u_{n} \in W^{1,1}\left(A;\mathbb{R}^{d}\right),
\chi_{n} \in BV\left(  A;\left\{  0,1\right\}  \right),  
\nonumber\\
&  \hspace{3cm} \left.  u_{n}\to u\text{ in }
L^{1}\left(  A;\mathbb{R}^{d}\right),
\chi_{n}\overset{\ast}{\rightharpoonup}\chi\text{ in }
BV\left(A;\left\{  0,1\right\}  \right)  \right\},\nonumber
\end{align}
and they showed that $\mathcal{F}_{\mathcal{OD}}\left(\chi,u;\cdot\right)$ is the trace on the open subsets of $\Omega$ of a 
finite Radon measure. The characterisation of this measure was provided by obtaining an integral representation for
$\mathcal{F}_{\mathcal{OD}}\left(\chi,u;A\right)$, where $A$ is an open subset of $\Omega$.

The case of non-convex $W_i$ with superlinear growth was addressed in the context of thin films in \cite{CZ1}. 

Our aim in this paper is to study the above optimal design problem within the context of non-standard growth
conditions. We take $\Omega$ to be a bounded, open subset of $\mathbb R^N$ and we let 
\begin{equation}\label{pqn}
1<p\leq q< \frac{N p}{N-1}.
\end{equation} 
If $N=1$ we let $1<p\leq q< +\infty.$
Let $F:BV\left(\Omega;\left\{  0,1\right\}\right)  \times 
W^{1,p}\left(\Omega;\mathbb{R}^{d}\right)
\rightarrow\left[  0,+\infty\right)$ 
be given by
\begin{equation}\label{F}
F\left(\chi,u\right)  :=\int_{\Omega}\chi\left(x\right) 
W_{1}\left(\nabla u(x)\right)  + \left(1-\chi\left(x\right)\right)  
W_{2}\left(\nabla u(x)\right) \, dx
+\left\vert D\chi\right\vert \left(\Omega\right)
\end{equation}
where $W_{i}:\mathbb{R}^{d\times N}\rightarrow\mathbb{R}$, $i=1,2$, are
continuous functions satisfying the following growth condition
\begin{equation}\label{growth}
\exists \beta > 0 :
0 \leq W_{i}\left(  \xi\right)  \leq
\beta\left(  1+\left\vert \xi\right\vert ^{q}\right),
\; \; \; \forall \xi \in \Rb^{d \times N}.  
\end{equation}

\noindent To simplify the notation, in the sequel, we let 
$f: \{0,1\} \times \mathbb R^{d \times N} \to \mathbb R$
be defined as
\begin{equation}\label{density}
f\left(  b,\xi\right)  :=b  W_1(\xi)+ (1-b)W_2(\xi). 
\end{equation}

\noindent In order to localize \eqref{F} we set, for every open set $A \subset \Omega$ and every 
$(\chi ,u) \in BV(A;\{0,1\})\times W^{1,p}(A;\mathbb R^d)$,
\begin{equation}
\label{FchiuA}
F(\chi, u; A):= \int_A f(\chi (x), \nabla u(x))\, dx + |D \chi|(A),
\end{equation}
and 
we define the relaxed functionals
\begin{align}
\mathcal{F}\left(  \chi,u;A\right)   &  :=\inf\left\{  
\underset{n\rightarrow +\infty}{\lim\inf}\,
F\left(\chi_{n},u_{n};A\right)  :  u_{n}
\in W^{1,q}\left(A;\mathbb{R}^{d}\right),
\chi_{n}
\in BV\left(  A;\left\{  0,1\right\}  \right),  
\right. \label{introrelaxed}\\
&  \hspace{3cm} \left.  u_{n}\rightharpoonup u\text{ in }
W^{1,p}\left(  A;\mathbb{R}^{d}\right),
\chi_{n}\overset{\ast}{\rightharpoonup}\chi\text{ in }
BV\left(A;\left\{  0,1\right\}  \right)  \right\}\nonumber
\end{align}
and
\begin{align}
\mathcal{F}_{\operatorname*{loc}}\left(  \chi,u;A\right)   &  
:=\inf\left\{
\underset{n\rightarrow +\infty}{\lim\inf}\, F\left(  \chi_{n},u_{n};A\right)
: u_{n}  \in W_{\operatorname*{loc}}^{1,q}
\left(A;\mathbb{R}^{d}\right),
 \chi_{n}  \in 
BV\left(A;\left\{0,1\right\}  \right),  \right. \label{introrelaxedloc}\\
&  \hspace{3cm} \left.  u_{n}\rightharpoonup u\text{ in }W^{1,p}
\left(  A;\mathbb{R}^{d}\right), 
\chi_{n}\overset{\ast}{\rightharpoonup}\chi\text{ in }
BV\left(A;\left\{  0,1\right\}  \right)\right\}. \nonumber
\end{align}
Notice that in these functionals there is a gap between the space of admissible
macroscopic fields $u \in W^{1,p}\left(A;\mathbb{R}^{d}\right)$ and the smaller
space $W^{1,q}\left(A;\mathbb{R}^{d}\right)$ where the growth hypothesis \eqref{growth}
ensures boundedness of the energy. Thus, assuming also that
the following coercivity condition holds
\begin{equation}\label{coer}
\exists \alpha > 0 : W_i(\xi) \geq \alpha |\xi|^p, \; \; \; \forall \xi \in \Rb^{d \times N},
\end{equation}
sequences of deformations
$u_n \in W^{1,q}\left(A;\mathbb{R}^{d}\right)$ that have bounded energy will be weakly compact in
$W^{1,p}\left(A;\mathbb{R}^{d}\right)$, but not necessarily in $W^{1,q}\left(A;\mathbb{R}^{d}\right)$,
so it may be possible to energetically approach functions $u \in W^{1,p} \setminus W^{1,q}$ and the aim of the 
above relaxed functionals is to provide the effective energy associated to $u$.

We recall that a sequence $\chi_n \to \chi$ in BV weak * if and only if $\chi_n$ is bounded
in BV and $\chi_n \to \chi$ in the strong topology of $L^1$. Due to this fact, and to the
expression of the energy \eqref{F}, in the above functionals we could also have taken
$\chi_n \to \chi$ in $L^1\left(A;\left\{  0,1\right\}  \right)$, obtaining in each case
the same infimum. However, replacing the $W^{1,p}$ weak convergence of the sequence $u_n$
by its strong convergence in $L^p$ does not yield the same infimum, unless a coercivity 
condition of the type \eqref{coer} is considered.

We will also consider the case $p=1$ and 
$1 \leq q < \frac{N}{N-1}$ ($1 \leq q < + \infty$ when $N=1$),
so we define
\begin{align}
\mathcal{F}_1\left(\chi,u;A\right)   &  :=\inf\left\{  
\underset{n\rightarrow +\infty}
{\lim\inf}\, F\left(\chi_{n},u_{n};A\right)  :
u_{n}\in W^{1,q}\left(  A;\mathbb{R}^{d}\right),
\chi_{n}\in BV\left(  A;\left\{  0,1\right\}  \right),\right. 
\label{relaxed1}\\
&  \hspace{3cm} \left.  u_{n}\overset{\ast}{\rightharpoonup} u
\text{ in } BV\left(A;\mathbb{R}^{d}\right),
\chi_{n}\overset{\ast}{\rightharpoonup}\chi
\text{ in } BV\left(A;\left\{0,1\right\}  \right)  \right\}, \nonumber
\end{align}
and, analogously, by replacing $W^{1,q}$ by $W^{1,q}_{\rm loc}$,
$
\mathcal{F}_{1, {\rm loc}}\left(\chi,u;A\right).
$
Contrary to what happens with regard to the topology used for the 
convergence of the sequence $\chi_n$,
as in the case $p>1$, taking $L^1$ strong or BV weak * convergence for the sequence $u_n$ does not give 
the same infimum, unless a coercivity condition of the type \eqref{coer} is considered.

Similar functionals were considered by Soneji in \cite{Soneji} and 
\cite{S1} in the case where there is
no dependence on the field $\chi$.

Our goal in this paper is to investigate whether the functionals \eqref{introrelaxed}, \eqref{introrelaxedloc}, and their variants when $p=1$, can be represented by certain
Radon measures defined on the open subsets of $\overline{\Omega}$ and, if so, whether a characterisation
of these measures can be obtained. As it turns out a (strong) measure representation is only true for \eqref{introrelaxedloc}
(see Theorem \ref{measureFloc} where some information on the corresponding measure is provided under some convexity assumptions), whereas \eqref{introrelaxed} only admits a weak measure representation
(cf. Definition \ref{measrep} and Theorem \ref{1<p<q}). In the one
dimensional case we provide a characterisation
of these measures (see Proposition \ref{lb1D}). 

In the case independent of the field $\chi$, and when $p=q$, it is well known that
\begin{align*}
\mathcal{F}(u;A) = & \inf\left\{  \underset{n\rightarrow +\infty}{\liminf}
\int_{\Omega}f(\nabla u_n(x)) \, dx  :  u_{n}
\in W^{1,q}\left(A;\mathbb{R}^{d}\right), 
 u_{n}\rightharpoonup u\text{ in }
W^{1,p}\left(  A;\mathbb{R}^{d}\right) \right\} \\
= & \int_{\Omega}Qf(\nabla u(x)) \, dx, 
\end{align*}
where $Qf$ denotes the quasiconvex envelope of $f$ (see Definition \ref{qcxenv}).
If $q > \frac{Np}{N-1}$ it may happen that $\mathcal{F}(u;\Omega) = 0$ (see \cite{BM}), and
if $q = \frac{Np}{N-1}$ then $\mathcal{F}(u;\cdot)$ may not even be subadditive (see \cite{ADM}).

Assuming the above growth and coercivity hypotheses on the density $f$, 
\begin{equation}\label{gc}
\alpha |\xi|^p \leq f(\xi) \leq \beta (1 + |\xi|^q),
\end{equation}
and relying on the existence of
a trace-preserving linear operator from $W^{1,p}$ to $W^{1,p}$ which improves the integrability of
$u$ and $\nabla u$ (see Lemma \ref{Lemma2.4FMy}), it was shown by Fonseca and Mal\'y
in \cite{FMy} that the functionals
\begin{equation}\nonumber
\mathcal{F}^{q,p}\left(u;A\right)  :=\inf\left\{  
\underset{n\rightarrow +\infty}{\lim\inf}\,
\int_{\Omega}f(\nabla u_n(x)) \, dx  :  u_{n}
\in W^{1,q}\left(A;\mathbb{R}^{d}\right), 
 u_{n}\rightharpoonup u\text{ in }
W^{1,p}\left(  A;\mathbb{R}^{d}\right) \right\}
\end{equation}
and
$\mathcal{F}^{q,p}_{{\rm loc}}\left(u;A\right)$ (defined as above just replacing $W^{1,q}$ by $W^{1,q}_{\rm loc}$) 
when finite, admit a weak measure representation in the first case, and a strong measure
representation in the second. In the latter case, denoting by $\mu_a$
the density of the corresponding measure $\mu$ with respect to the Lebesgue measure, they also showed
that 
$$\mu_a (x_0) \geq Qf(\nabla u (x_0)), \; \; \forall u \in 
W^{1,p}\left(\Omega;\mathbb{R}^{d}\right), \; \; {\rm a.e.} \; \; x_0 \in \Omega.$$
The reverse inequality was obtained in \cite{BFM}, thus fully identifying the bulk part
of the measure $\mu$. 

We also refer to Acerbi, Bouchitt\'e and Fonseca \cite{ABF} where the case of inhomogeneous densities
$h(x,\xi)$ is treated. Assuming convexity of $h$ with respect to the second variable,
as well as a growth condition of the type \eqref{gc}, it is shown that
$$ \mathcal{F}^{q,p}_{{\rm loc}}\left(u;A\right) = \int_A h(x,\nabla u(x)) \, dx + \mu_s(u;A),$$
where $\mu_s(u,\cdot)$ is a non-negative Radon measure, singular with respect to the Lebesgue measure.

Finally, we mention the results of Coscia and Mucci \cite{CM} and 
Mucci \cite{Mucci}. In their work, the authors consider relaxation and integral representation of integral functionals of the type $\displaystyle \int_{\Omega}h(x,\nabla u(x)) \, dx$, when $h$ satisfies
\begin{equation}
\label{pxgrowth1}
|\xi|^{p(x)} \leq h(x,\xi) \leq C(1 + |\xi|^{p(x)}),
\end{equation}
so that the integrability exponent $p(x)$ of the admissible fields depends in a continuous or regular piecewise continuous way on the location in the body (see \cite[assumptions (2.9) and (2.1) and Definitions 2.1 and 2.2]{CM}, respectively). We point out that
the bulk energy density in \eqref{F} can be seen as satisfying a 
generalization of \eqref{pxgrowth1} to the case of a discontinuous 
exponent $p(x)$, we refer to  Remark \ref{remCosciaMuccietalia} for
more details.

\smallskip

The paper is organized as follows. In Section 2 we set the notation 
and we provide some definitions and results
which will be used throughout the paper, whereas in
Section 3 we prove some auxiliary results. The main representation 
theorems are proved in Section 4, where we provide
a partial characterisation of the measures which represent 
\eqref{introrelaxed} and \eqref{introrelaxedloc} 
in the convex case and a full characterisation 
in the one dimensional setting, we also obtain a sufficient 
condition for lower semicontinuity of \eqref{F}. Finally, in Section 
5 we give an application to a 3$D$-1$D$ dimension reduction problem.

\section{Preliminaries}

In this section we fix notations and quote some definitions and 
results dealing with sets of finite perimeter and several notions of 
quasiconvexity that will be used in the sequel.

Throughout the text $\Omega \subset \mathbb R^{N}$ will denote an open, bounded set.

We will use the following notations:
\begin{itemize}
\item ${\mathcal O}(\Omega)$ is the family of all open subsets
of $\Omega $;
\item $\mathcal M (\Omega)$ is the set of finite Radon
measures on $\Omega$;
\item $\mathcal L^{N}$ and $\mathcal H^{N-1}$ stand for the  $N$-dimensional Lebesgue measure 
and the $\left(  N-1\right)$-dimensional Hausdorff measure in $\mathbb R^N$, respectively;
\item $\left |\mu \right |$ stands for the total variation of a measure  $\mu\in \mathcal M (\Omega)$; 
\item the symbol $d x$ will also be used to denote integration with respect to $\mathcal L^{N}$;
\item  $C$ represents a generic positive constant that may change from line to line.
\end{itemize}

We start by recalling a well known result due to Ioffe 
\cite[Theorem 1.1]{Ioffe}.
\begin{theorem}\label{thm2.4ABFvariant}
Let $g:\mathbb R^m\times \mathbb R^{d \times N} \to [0, +\infty)$ 
be a Borel integrand such that $g(b,\cdot)$ is convex for every 
$b \in \mathbb R^m$. Then the functional 
$$
G(v,u):=\int_{\Omega}g(v(x), \nabla u(x)) \, dx
$$
is lower semicontinuous in 
$L^1(\Omega;\mathbb R^m)_{\rm strong} \times 
W^{1,1}(\Omega;\mathbb R^d)_{\rm weak}$. 
\end{theorem}


In the following we give some preliminary notions related with sets of finite
perimeter. For a detailed treatment we refer to \cite{AFP}.

	A function $w\in L^{1}(\Omega;{\mathbb{R}}^{d})$ is said to be of
	\emph{bounded variation}, and we write $w\in BV(\Omega;{\mathbb{R}}^{d})$, if
	all its first order distributional derivatives $D_{j}w_{i}$ belong to $\mathcal{M}%
	(\Omega)$ for $1\leq i\leq d$ and $1\leq j\leq N$.

The matrix-valued measure whose entries are $D_{j}w_{i}$ is denoted by $Dw$
and $|Dw|$ stands for its total variation.
We observe that if $w\in BV(\Omega;\mathbb{R}^{d})$ then $w\mapsto|Dw|(\Omega)$ is lower
semicontinuous in $BV(\Omega;\mathbb{R}^{d})$ with respect to the
$L_{\mathrm{loc}}^{1}(\Omega;\mathbb{R}^{d})$ topology.

\begin{definition}
	\label{Setsoffiniteperimeter} Let $E$ be an $\mathcal{L}^{N}$- measurable
	subset of $\mathbb{R}^{N}$. For any open set $\Omega\subset\mathbb{R}^{N}$ the
	{\em perimeter} of $E$ in $\Omega$, denoted by $P(E;\Omega)$, is the variation of
	$\chi_{E}$ in $\Omega$, i.e.
	\begin{equation}
	\label{perimeter}P(E;\Omega):=\sup\left\{  \int_{E} \mathrm{div}\varphi(x) \,dx:
	\varphi\in C^{1}_{c}(\Omega;\mathbb{R}^{d}), \|\varphi\|_{L^{\infty}}%
	\leq1\right\}  .
	\end{equation}
	We say that $E$ is a {\em set of finite perimeter} in $\Omega$ if $P(E;\Omega) <+
	\infty.$
\end{definition}

Recalling that if $\mathcal{L}^{N}(E \cap\Omega)$ is finite, then $\chi_{E}
\in L^{1}(\Omega)$, by \cite[Proposition 3.6]{AFP}, it follows
that $E$ has finite perimeter in $\Omega$ if and only if $\chi_{E} \in
BV(\Omega)$ and $P(E;\Omega)$ coincides with $|D\chi_{E}|(\Omega)$, the total
variation in $\Omega$ of the distributional derivative of $\chi_{E}$.
Moreover,  a
generalized Gauss-Green formula holds:
\begin{equation}\nonumber
{\int_{E}\mathrm{div}\varphi(x) \, dx
=\int_{\Omega}\left\langle\nu_{E}(x),\varphi(x)\right\rangle \, d|D\chi_{E}|,
\;\;\forall\,\varphi\in C_{c}^{1}(\Omega;\mathbb{R}^{d})},
\end{equation}
where $D\chi_{E}=\nu_{E}|D\chi_{E}|$ is the polar decomposition of $D\chi_{E}$.

We also recall that, when dealing with sets of finite measure, a sequence of
sets $E_{n}$ converges to $E$ in measure in $\Omega$ if $\mathcal{L}%
^{N}(\Omega\cap(E_{n}\Delta E))$ converges to $0$ as $n\rightarrow +\infty$,
where $\Delta$ stands for the symmetric difference. 
This convergence is equivalent to $L^{1}(\Omega)$ 
convergence of the characteristic functions of the corresponding 
sets.

In order to compare several notions of quasiconvexity, we start by 
recalling the one introduced by Morrey.
\begin{definition}\label{Morrey-qcx}
A Borel measurable and locally bounded function $f:\mathbb R^{d\times N}\to \mathbb R$ is said to be 
{\em quasiconvex} if
\begin{equation}\label{Mqcx}
f(\xi)\leq\frac{1}{{\mathcal L}^N(D)}\int_D 
f(\xi+\nabla \varphi(x)) \, dx,
\end{equation}
for every bounded, open set $D\subset \mathbb R^N$, for every $\xi \in \mathbb R^{d\times N}$ and for 
every $\varphi \in W^{1,\infty}_0(D;\mathbb R^d)$.  
\end{definition}

\begin{remark}\label{Mqcxobservation}
{\rm We recall that if \eqref{Mqcx} holds for a certain set $D$, then it holds for any bounded, open set 
in $\mathbb R^N$.
Notice also that, in the above definition, the value $+\infty$ is excluded from the range of $f$.}
\end{remark}

The following notion of $W^{1,p}$-quasiconvexity was introduced by Ball and Murat in \cite{BM}.

\begin{definition}\label{W1pqcxBM}
Let $p\in [1,+\infty]$ and let $f:\mathbb R^{d\times N}\to \mathbb (-\infty,+\infty]$ be Borel measurable 
and bounded below. $f$ is said to be $W^{1,p}$-{\em quasiconvex}  if
\begin{equation}\label{BMqcxineq}
f(\xi)\leq\frac{1}{{\mathcal L}^N(D)}\int_D f(\xi+\nabla \varphi(x)) \, dx,
\end{equation}
for every bounded, open set $D\subset \mathbb R^N$, with ${\mathcal L}^N(\partial D)=0$,  
for every $\xi \in \mathbb R^{d\times N}$ and for every $\varphi \in W^{1,p}_0(D;\mathbb R^d)$. 
\end{definition}

\begin{remark}\label{BMremark}
{\rm Definition \ref{W1pqcxBM} coincides with quasiconvexity for finite functions $f$ and when $p=+\infty$.
	
\smallskip
If $f$ is $W^{1,p}$- quasiconvex, then it is $W^{1,q}$-quasiconvex for all $q$ with $p\leq q\leq +\infty$, so 
$W^{1,1}$-quasiconvexity is the strongest notion and $W^{1,\infty}$-quasiconvexity is the weakest notion.
	
\smallskip
	
The following facts were established in \cite[Proposition 2.4]{BM}.
If $f$ is bounded below, upper semicontinuous and
\begin{equation}\label{qgrowthabove}
f(\xi)\leq K(1+|\xi|^q), \hbox{ for every } \xi \in \mathbb R^{d\times N},
\end{equation}
then $f$ is $W^{1,q}$-quasiconvex if and only if it is $W^{1,\infty}$-quasiconvex.
	
\smallskip
\noindent If there exist $K_1\in (0,+\infty)$ and $K_2 \in \Rb$ such that
\begin{equation}\label{pgrowthbelow}
 K_1|\xi|^p + K_2 \leq f(\xi), \hbox{ for every } \xi \in \mathbb R^{d\times N},
\end{equation}
for $p \in [1,+\infty)$, then $f$ is $W^{1,1}$-quasiconvex if and only if it is $W^{1,p}$-quasiconvex.}
\end{remark}

It is known that the above notions, when applied to 
$f(b,\cdot) := bW_1(\cdot) + (1-b)W_2(\cdot)$, are necessary and sufficient conditions for lower 
semicontinuity of the functional $F(\chi, u)$ in \eqref{F}, when $f$ satisfies both \eqref{pgrowthbelow} and 
\eqref{qgrowthabove} with the same exponent $p=q >1$.
The sufficiency of $W^{1,p}$- quasiconvexity for the lower semicontinuity of $F(\chi,u)$ is no longer true in the case 
$p < q$, see Gangbo \cite{G} for an example.
 
The notion of closed $W^{1,p}$-quasiconvexity, which we recall next, 
was introduced by Pedregal in \cite{P} and used by
Kristensen in \cite{K} in order to provide sufficient conditions for 
the lower semicontinuity of integral functionals of the type 
$$\int_{\Omega}f(\nabla u(x)) \, dx,$$ 
assuming that $f$ satisfies \eqref{pgrowthbelow}. 
It requires the use of Young measures, for more details we refer to \cite[Theorem 2.2]{FMAq}. 

Let ${\mathcal M}^p$ be the set of probability measures $\mu $ on $\mathbb R^{d \times N}$ with finite $p$-th moment, 
i.e. such that 
$\left\langle\mu, |\cdot|^p\right\rangle \; < +\infty,$ when 
$p < +\infty$, and with bounded support when $p= +\infty$, and such 
that the following Jensen's inequality,
\begin{equation}\label{Jensen}
\int_{\mathbb R^{d \times N}} f(\xi) \, d \mu \geq f(\overline \mu), \hbox{ where } \overline \mu:= \left\langle\mu, {\rm id}\right\rangle,
\end{equation}
holds for all $W^{1,\infty}$-quasiconvex functions $f:\mathbb R^{d \times N}\to \mathbb R$ for which there exists 
a constant $c=c(f)$ such that  $|f(\xi)|\leq C(1+|\xi|^p)$. 

The Young measure $\mu$ is said to be homogeneous if there is a Radon measure $\mu_0 \in  {\mathcal M}(\mathbb R^d)$ 
such that $\mu_x = \mu_0$ for a.e. $x \in \Omega$.

In \cite{KP1, KP2}, the set ${\mathcal M}^p$ is shown to coincide with the homogeneous $W^{1,p}$-gradient Young measures.
We recall that a Young measure $\mu$ is called a gradient Young measure if it is generated by a
sequence of gradients, more precisely, $\mu$ is a $W^{1,p}$- gradient Young measure if it is
generated by $\nabla u_n$ and $u_n\rightharpoonup u$ in $W^{1,p}(\Omega;\mathbb R^d)$. 

Also for each $p \in (1,+\infty)$,
$$
{\mathcal M}^1\cap \{\mu: \int_{\mathbb R^{d \times N}}|\xi|^p \, d \mu(\xi) < +\infty\}= {\mathcal M}^p.
$$

\begin{definition}\label{closedW1pqcx}
The integrand $f:\mathbb R^{d \times N} \to \mathbb R \cup \{+\infty\}$ is said to be {\em closed $W^{1,p}$-quasiconvex}, 
if it is lower semicontinuous and Jensen's inequality \eqref{Jensen} holds for all $\mu \in {\mathcal M}^p$. 
\end{definition}

\begin{definition}\label{qcxenv}
The {\em quasiconvex envelope} (respectively, {\em $W^{1,p}$-quasiconvex envelope}, 
{\em closed $W^{1,p}$-qua\-si\-con\-vex envelope}) of $f$ is the greatest
quasiconvex (respectively, $W^{1,p}$-quasiconvex, closed $W^{1,p}$-qua\-si\-con\-vex) function that 
is less than or equal to $f$.
\end{definition} 

The following result (see \cite[Proposition 2.5]{FMAq}) will be used in the sequel.

\begin{proposition}\label{FMYoungmeasures}
If $v_n$ generates a Young measure $\nu$ and $u_n \to u$ in measure, then the pair $(u_n,v_n)$ generates the 
Young measure $\mu$ defined as
$$
\mu_x:= \delta_{u(x)}\otimes \nu_x, \hbox{ for a.e. }x \in \Omega.
$$
\end{proposition}

We conclude by recalling the notions of weak and strong representations by means of measures.

\begin{definition}\label{measrep}
	Let $\mu$ be a Radon measure on $\overline{\Omega}.$ We say that
	
	\begin{enumerate}
		\item[a)] $\mu$ {\em (strongly) represents} 
		$\mathcal{F}\left(  \chi,u;\cdot\right)$ 
		if $\mu\left(  A\right)  =\mathcal{F}\left(  \chi,u;A\right)$ 
		for all open sets $A\subset\Omega;$
		
		\item[b)] $\mu~$ {\em weakly represents} 
		$\mathcal{F}\left(  \chi,u;\cdot\right)$
		if $\mu\left(  A\right)  \leq\mathcal{F}\left(\chi,u;A\right)  
		\leq \mu\left(  \overline{A}\right)$ for all open sets $A\subset\Omega.$
	\end{enumerate}
\end{definition}

\section{Auxiliary Results}

In this section we prove some auxiliary results which will be used in the sequel
in the proofs of our representation theorems. We begin by showing that, under
the coercivity assumption (\ref{coer})
the infima in \eqref{introrelaxed} and \eqref{introrelaxedloc} are attained.

\begin{proposition}\label{attain}
Let $f$ be as in \eqref{density} where $W_i$, $i = 1,2$, satisfy the growth condition \eqref{growth},
as well as the coercivity condition \eqref{coer}. Let $\mathcal F(\chi,u;\Omega)$ be as in
\eqref{introrelaxed}. 
\begin{itemize}
\item[i)] If $\mathcal F(\chi,u;\Omega) < + \infty$, then it is attained, that is, there exist sequences
$u_n \in W^{1,q}(\Omega;\Rb^d)$, $\chi_n \in BV(\Omega;\{0,1\})$ such that
$u_n \weak u$ in $W^{1,p}(\Omega;\Rb^d)$, $\chi_n \weakstar \chi$ in $BV(\Omega;\{0,1\})$ and
$$\mathcal F(\chi,u;\Omega) = \lim_{n \to + \infty}F(\chi_n,u_n;\Omega).$$
\item[ii)] If $u_n \in W^{1,q}(\Omega;\Rb^d)$, $\chi_n \in BV(\Omega;\{0,1\})$ are
sequences such that $u_n \weak u$ in $W^{1,p}(\Omega;\Rb^d)$, $\chi_n \weakstar \chi$ in $BV(\Omega;\{0,1\})$
and if $\mathcal F(\chi_n,u_n;\Omega) < + \infty$, $\forall n$, then
$$\mathcal F(\chi,u;\Omega) \leq \liminf_{n \to + \infty}\mathcal F(\chi_n,u_n;\Omega).$$
\item[iii)] In the case $p=1$, $1 \leq q < \frac{N}{N-1}$ the above results also hold for
\eqref{relaxed1} under the assumption $W_i(\xi) \geq \alpha |\xi|$, for some $\alpha > 0$ and for all 
$\xi \in \Rb^{d \times N}$.
\end{itemize}
Similar statements hold for the functionals $\mathcal F_{\rm loc}$
given in \eqref{introrelaxedloc} and $\mathcal F_{1, \rm loc}$. 
\end{proposition}
\begin{proof}
i) By definition of $\mathcal F(\chi,u;\Omega)$, for each $n \in \Nb$, there exist sequences
$u_{n,k} \in W^{1,q}(\Omega;\Rb^d)$ and $\chi_{n,k} \in BV(\Omega;\{0,1\})$
such that $u_{n,k} \weak u$ in $W^{1,p}(\Omega;\Rb^d)$ as $k \to + \infty$, $\chi_{n,k} \weakstar \chi$ 
in $BV(\Omega;\{0,1\})$ as $k \to + \infty$ and
$$\liminf_{k \to + \infty}F(\chi_{n,k},u_{n,k}) < \mathcal F(\chi,u;\Omega) + \frac{1}{n}.$$
By taking subsequences if necessary, for each $n$, we may assume that
$$\liminf_{k \to + \infty}F(\chi_{n,k},u_{n,k}) = \lim_{k \to + \infty}F(\chi_{n,k},u_{n,k}).$$
Hence for each $n$, there exists $k_n^1$ such that
\begin{equation}\label{ineq1}
F(\chi_{n,k},u_{n,k}) < \mathcal F(\chi,u;\Omega) + \frac{1}{n}, \; \; \; \forall k \geq k_n^1.
\end{equation}
Since $u_{n,k} \weak u$, as $k \to + \infty$, in $L^p(\Omega;\Rb^d)$, the sequence $u_{n,k}$ is bounded in 
$L^p(\Omega;\Rb^d)$ 
so, by the metrizability of the unit ball of $L^p(\Omega;\Rb^d)$ in the weak topology, there exists a metric 
$d$ such that $d( u_{n,k},u) \to 0$ as $k \to + \infty$.
Thus, there exists $k_n^2$ such that
$$ d(u_{n,k}, u) < \frac{1}{n}, \; \; \;
\|\chi_{n,k} - \chi \|_{L^1(\Omega;\{0,1\})} < \frac{1}{n}, \; \; \; \forall k \geq k_n^2.$$
Choose an increasing sequence $k_n \in \Nb$ such that $k_n \geq \max\{k_n^1,k_n^2\}$, then
$ d(u_{n,k_n}, u) \to 0$,
so $u_{n,k_n} \weak u$ in $L^p(\Omega;\Rb^d)$, and $\chi_{n,k_n} \to \chi$ in $L^1(\Omega;\{0,1\})$, 
as $n \to + \infty$. By the coercivity condition \eqref{coer} and \eqref{ineq1} it follows that
$$\sup_{n}\left [\alpha \int_{\Omega}|\nabla u_{n,k_n}(x)|^p \, dx + |D \chi_{n,k_n}|(\Omega)\right ]
\leq \sup_{n} F(\chi_{n,k_n},u_{n,k_n};\Omega) \leq \mathcal F(\chi,u;\Omega) + 1 < + \infty$$
so $u_{n,k_n}$ is bounded in $W^{1,p}(\Omega;\Rb^d)$ and $\chi_{n,k_n}$ is bounded in 
$BV(\Omega;\{0,1\})$ and therefore these sequences are admissible for $\mathcal F(\chi,u;\Omega)$.
Hence, by \eqref{ineq1},
$$\mathcal F(\chi,u;\Omega) \leq \liminf_{n \to + \infty}F(\chi_{n,k_n},u_{n,k_n};\Omega)
\leq \limsup_{n \to + \infty}\left(\mathcal F(\chi,u;\Omega) + \frac{1}{n}\right) = \mathcal F(\chi,u;\Omega),$$
which, passing to another subsequence if necessary, proves i).

\smallskip

ii) Follows immediately from part i) and a standard diagonalization argument.

\smallskip

iii) The proof in the case $p=1$ is similar to the case $p>1$ replacing the weak topology in $W^{1,p}$
by the weak * topology in BV. 
\end{proof}

In order to prove our measure representation results 
we will need the following lemmas (see Lemmas 2.4 and 3.4 in \cite{FMy} for the case \eqref{pqn}
and Lemma 5.4 in \cite{Soneji} for the case $p=1$).

\begin{lemma}\label{Lemma2.4FMy} Let $p$ and $q$ satisfy \eqref{pqn}.
Let $V\subset\subset\Omega,~W\subset\Omega$ be open
sets such that $\Omega=V\cup W$, and let 
$v\in W^{1,q}\left(V;\mathbb{R}^{d}\right)$,
$w \in W^{1,q}\left(W;\mathbb{R}^{d}\right).$
Then, for every $m\in\mathbb{N}$, there exist 
$z\in W^{1,q}\left(\Omega;\mathbb{R}^{d}\right)$ and open sets 
$V^{\prime}\subset V$ and $W^{\prime}\subset W,$ 
such that $V^{\prime}\cup W^{\prime}=\Omega,$
$z=v$ in $\Omega\backslash W'$, $z=w$ in
$\Omega\backslash V^{\prime},$
\begin{equation}
{\cal L^N}(V' \cap W')\leq \frac{C}{m}\label{stripestimation}%
\end{equation}
and
\begin{equation}\label{fundest}
\begin{array}{ll}
\displaystyle{\left\Vert z\right\Vert _{W^{1,q}(  V'\cap W')}
\leq
\displaystyle{\frac{C}{m^{\tau}}}\left(\left\Vert v
\right\Vert _{W^{1,p}(  V\cap W)}
+\left\Vert w\right\Vert _{W^{1,p}\left(  V\cap W\right)}
+m\left\Vert w-v\right\Vert _{L^{p}(  V\cap W) }\right)},
\end{array}
\end{equation}
where $C=C\left(p,q,V,W\right)$ and $\tau=\tau\left(N,p,q\right) >0.$
\end{lemma}

\begin{lemma}\label{Lemma 3.4 FMy} 
Let $V,\ W\subset\Omega$ be open sets such that
$V\subset\subset\Omega$ and $\Omega=V\cup W$. 
Let $f$ be as in \eqref{density} where $W_i$, $i = 1,2$, satisfy the growth condition \eqref{growth}.
If 
$u\in W^{1,p}\left(\Omega;\mathbb{R}^{d}\right)$ and 
$\chi\in BV\left(\Omega;\left\{0,1\right\}\right)$, then
\[
\mathcal{F}\left(\chi,u;\Omega\right)  \leq
\mathcal{F}\left(\chi,u;V\right)  + 
\mathcal{F}\left(\chi,u;W\right).  
\]
The same result holds for $\mathcal{F}_1\left(\chi,u;\Omega\right)$.
\end{lemma}

\begin{proof}

Choose $\varepsilon>0$ and let $V^{\prime}\subset V$,
$W^{\prime}\subset W$ be such that $\Omega=V^{\prime}\cup W^{\prime}$,
$\overline {V^{\prime}\cap W^{\prime}}\subset V\cap W$ and $V'$ and $W'$ have $C^1$ boundaries
so that we can apply Rellich's compact embedding theorem. Thus, by this result and the 
definition of the relaxed functionals $\mathcal{F}\left(\chi,u;V\right)$ and $\mathcal{F}\left(\chi,u;W\right)$, 
there exist 
$v_{n}\in W^{1,q}\left(V;\mathbb{R}^{d}\right)$,
$w_{n} \in W^{1,q}\left(W;\mathbb{R}^{d}\right)$,
$\chi_{n} \in BV\left(V;\left\{0,1\right\}\right)$ and  
$\zeta_{n} \in BV\left(W;\left\{0,1\right\}\right)$ 
such that
\[
\begin{array}[c]{lll}
v_{n}\rightharpoonup u\text{ in }W^{1,p}\left(V;\mathbb{R}^{d}\right),  &  &
w_{n}\rightharpoonup u\text{ in }W^{1,p}\left(W;\mathbb{R}^{d}\right),\\
\chi_{n}\overset{\ast}{\rightharpoonup}\chi\text{ in }
BV\left(V;\left\{0,1\right\}\right), &  & 
\zeta_{n}\overset{\ast}{\rightharpoonup}\chi\text{ in }
BV\left(W;\left\{0,1\right\}\right),
\end{array}
\]
$$
\left\Vert v_{n}-u\right\Vert _{L^{p}\left(V^{\prime}\cap W^{\prime}\right)}
\leq\frac{1}{n},
~\left\Vert w_{n}-u\right\Vert _{L^{p}\left(V^{\prime}\cap W^{\prime}\right)}
\leq\frac{1}{n},
$$
$$
\left\Vert \chi_{n}-\chi\right\Vert_{L^{1}\left(V^{\prime}\cap 
W^{\prime}\right)} \leq\frac{1}{n},
~\left\Vert\zeta_{n}-\chi\right\Vert_{L^{1}\left(V^{\prime}\cap 
W^{\prime}\right)}
\leq\frac{1}{n},
$$
and
\begin{align*}
\int_{V}f\left(\chi_{n}(x),\nabla v_{n}(x)\right) \, dx +
\left\vert D\chi_{n}\right\vert \left(V\right)   &  
\leq\mathcal{F}\left(\chi,u;V\right) + \varepsilon,\\
\int_{W}f\left(\zeta_{n}(x),\nabla w_{n}(x)\right) \, dx +
\left\vert D\zeta_{n}\right\vert \left(W\right)   &  
\leq\mathcal{F}\left(\chi,u;W\right) + \varepsilon.
\end{align*}
By virtue of Lemma \ref{Lemma2.4FMy}, there exist open sets 
$V_{n}\subset V^{\prime}$, $W_{n}\subset W^{\prime}$, 
and functions 
$z_{n}\in W^{1,q}\left(\Omega;\mathbb{R}^{d}\right)$, such that 
$V_{n}\cup W_{n}=\Omega,$ 
\begin{equation}\label{vn}
z_{n}  = v_{n} 
\text{ in }\Omega\backslash W_{n},\;\;\;
z_{n}  = w_{n} 
\text{ in }\Omega\backslash V_{n},
\end{equation}
and estimates of the type \eqref{stripestimation} and \eqref{fundest} hold. 

The sequence $z_n$ is admissible for ${\mathcal F}(\chi,u;\Omega)$
since $z_{n}\rightharpoonup u$ weakly in
$W^{1,p}\left(\Omega;\mathbb{R}^{d}\right)$. 
Indeed, by \eqref{stripestimation} and \eqref{vn}, $z_n$ converges to $u$ in measure, and by
\eqref{fundest}, 
$z_{n}$ is bounded in
$W^{1,p}\left(\Omega;\mathbb{R}^{d}\right)$, so it follows that, for a subsequence
which we do not relabel,
$z_{n} \rightharpoonup u$ in $W^{1,p}(\Omega;\mathbb R^d)$. 

We must now build a transition sequence $\eta_n$ between $\chi_n$ and 
$\zeta_n$ in the above sets $V_n$ and $W_n$, in such a way that an upper bound 
for the total variation of $\eta_n$ in $V_n\cap W_n$ is obtained.
In order to connect these functions without adding more interfaces, we argue 
as in \cite{BMMO} (see also \cite{CZ}). 
For $\delta > 0$ small enough, consider
$$
V'_\delta:=\{x \in W'\cap V': {\rm dist}(x, \partial \overline{V'}) < \delta\}.
$$
Given $x \in W$, let $d(x) := {\rm dist}(x;V')$. Since the distance function 
to a fixed set is Lipschitz continuous (see Exercise 1.1 in Ziemer \cite{Z}  and \cite{GP})
we can apply the change of variables formula (see Theorem 2, Section 3.4.3, 
in Evans and Gariepy \cite{EG}), to obtain
$$
\int_{V'_\delta \setminus \overline{V}}|\chi_n(x)-\zeta_n(x)| 
|{\rm det}\nabla d(x)| \, dx =
\int_0^\delta \left[ \int_{d^{-1}(y)} |\chi_n(x)-\zeta_n(x)| \,
d {\mathcal H}^{N-1}(x) \right] dy$$
and, as $|{\rm det}\nabla d(x)|$ 
is bounded and $\chi_n - \zeta_n \to 0$  in $L^1(W' \cap V';\mathbb R^d)$, it follows that, for almost every $\rho \in [0; \delta]$,
\begin{equation}\label{413bis}
\lim_{n\to +\infty}\int_{d^{-1}(\rho)}|\chi_n(x) - \zeta_n(x)| \, 
d{\mathcal H}^{N-1}(x) = 
\lim_{n\to +\infty}\int_{\partial V'_\rho}|\chi_n(x)-\zeta_n(x)| \,
d {\mathcal H}^{N-1}(x) = 0.
\end{equation}
Fix $\rho_0\in [0; \delta]$ such that \eqref{413bis} holds. We
observe that $V'_{\rho_0}$ is a set with locally Lipschitz boundary since 
it is a level set of a Lipschitz function (see, for example, Evans 
and Gariepy \cite{EG}). Hence, for every $n$ and $V_n\subset V'$ and $W_n\subset W'$, 
we can consider $\chi_n, \zeta_n$ on $\partial V_{n,\rho_n}$ in
the sense of traces and define
\begin{equation*}
\eta_n =\left\{\begin{array}{ll}
\chi_n &\hbox{ in } \overline{V}_{n,\rho_n},\\
\zeta_n &\hbox{ in } W_n\setminus \overline{V}_{n,\rho_n},
\end{array}\right.
\end{equation*}
and 
$$\int_{\partial V_{n,\rho_n}}|\chi_n(x)-\zeta_n(x)| \, d{\cal H}^{N-1}(x)
\leq \frac{C}{n},
$$
for a suitable constant $C$.

By the choice of $\rho_0$, $\eta_n$ is admissible for 
${\mathcal F}(\chi,u;\Omega)$.
In particular, $\eta_n \wsto \chi$ in $BV(\Omega;\{0,1\}).$
Using \eqref{growth},  \eqref{stripestimation} and \eqref{fundest} one has
\[
\int_{V_{n}\cap W_{n}}f\left(\eta_{n}(x),\nabla z_{n}(x)\right) \, dx 
\leq\beta\int_{V_{n}\cap W_{n}}
\left(1+\left\vert \nabla z_{n}(x)\right\vert ^{q}\right) \, dx 
\leq\frac{C}{n} + \frac{C}{n^{q\tau}},
\]
where $\tau$ is as in Lemma \ref{Lemma2.4FMy}. It follows that
\begin{align*}
\int_{\Omega}f\left(\eta_{n}(x),\nabla z_{n}(x)\right) \, dx
+\left\vert D\eta_{n}\right\vert \left(\Omega\right) & 
\leq\int_{V}f\left(\chi_{n}(x),\nabla v_{n}(x)\right) \, dx 
+\left\vert D\chi_{n}\right\vert \left(V\right)  \\
& +\int_{W}f\left(\zeta_{n}(x),\nabla w_{n}(x)\right) \, dx 
 +\left\vert D\zeta_{n}\right\vert \left(W\right) 
+ C(n^{-1}+ n^{-q\tau}+n^{-1}).
\end{align*}
Taking the limit in $n$ one obtains
\[
\mathcal{F}\left(\chi,u;\Omega\right)  \leq 
\mathcal{F}\left(\chi,u;V\right) + \mathcal{F}\left(\chi,u;W\right)  
+2\varepsilon,
\]
so, letting $\varepsilon\rightarrow 0^+$, we deduce the desired subadditivity inequality. 

Notice that in the case $p=1$ there is no need to consider the auxiliary sets $V'$ and $W'$ since,
by definition of $\mathcal{F}_1$, we have strong convergence in $L^1(\Omega;\Rb^d)$ of $v_n$ and $w_n$ to $u$. In
this case, to show admissibility of $z_n$ for $\mathcal{F}_1\left(\chi,u;\Omega\right)$ we conclude
first that $z_n$ is bounded in $W^{1,1}(\Omega;\Rb^d)$ and then use the embedding theorem to obtain 
$z_n \to u$ in $L^1(\Omega;\Rb^d)$.
\end{proof}

\begin{remark}{\rm
A similar result holds for ${\cal F}_{\rm loc}(\chi,u;\cdot)$ and ${\cal F}_{1,{\rm loc}}(\chi,u;\cdot)$, with an 
analogous proof. Indeed the proof relies on Lemma \ref{Lemma2.4FMy}, which can 
still be applied for ${\cal F}_{\rm loc}$ and ${\cal F}_{1,{\rm loc}}$, and a careful glueing argument for 
characteristic functions as was seen above.}
\end{remark}

\medskip

\section{Main Results}

Assume that $f$ is defined by \eqref{density} with $W_i$, $i=1,2$, 
satisfying \eqref{growth}. 
For every open set $A\subset \Omega\subset \subset \mathbb R^N$ and 
every 
$(\chi,u) \in BV(A;\{0,1\})\times W^{1,p}(A;\mathbb R^d)$ let $F$ be as 
in \eqref{FchiuA} and consider the relaxed functional 
${\cal F}(\chi,u;A)$ in 
\eqref{introrelaxed}. Then the following weak representation result holds. 

\begin{theorem}\label{1<p<q}
Let $f$ be given by $\left(  \ref{density}\right)$ and satisfy 
$\left(\ref{growth}\right)$, let $p,~q$ be as in (\ref{pqn})
and let $\chi\in BV\left(  \Omega;\left\{  0,1\right\}\right)$
and $u\in W^{1,p}\left(\Omega;\mathbb{R}^{d}\right).$ If 
$\mathcal{F}\left(\chi,u;\Omega\right)  < +\infty$, then there exists a 
non-negative Radon measure $\mu$ on $\overline{\Omega}$ which weakly 
represents $\mathcal{F}\left(\chi,u;\cdot\right).$ Likewise for $\mathcal{F}_1\left(\chi,u;\cdot\right).$
\end{theorem}

\begin{proof}
{\em Step 1.} We assume first that the coercivity condition \eqref{coer} holds.
Thus, by Proposition \ref{attain}, let 
$(\chi_n,u_n)\in BV(\Omega;\{0,1\}) \times 
W^{1,q}(\Omega;\mathbb R^d)$  
be a realizing sequence for 
${\cal F}(\chi,u;\Omega)$, that is,
$u_n \rightharpoonup u$ in $W^{1,p}(\Omega;\mathbb R^d)$,
$ \chi_n \overset{\ast}{\rightharpoonup} \chi$ in $BV(\Omega;\{0,1\})$ and
\begin{equation}\label{choice}
\displaystyle{\lim_{n \to +\infty} F(\chi_n,u_n;\Omega)
= {\cal F}(\chi,u;\Omega)}.
\end{equation}
Let $\mu_n = f(\chi_n(\cdot), \nabla u_n (\cdot)) {\cal L}^N\lfloor{\Omega}
+ |D \chi_n|(\Omega)$ and extend this sequence of measures outside of 
$\Omega$ by setting, for any Borel set $E \subset \mathbb R^N$,
$$\lambda_n(E) = \mu_n(E \cap \Omega).$$ 
Passing, if necessary, to a subsequence, we can assume that there exists a 
non-negative Radon measure $\mu$ (depending on $\chi$ and $u$) on 
$\overline{\Omega}$  such that $\lambda_n \wsto \mu$ in the sense of measures 
in $\overline {\Omega}$.
Let $\phi_k \in C_0(\overline{\Omega})$ be an increasing sequence of functions 
such that $0 \leq \phi_k \leq 1$ and $\phi_k(x) \to 1$ a.e. in 
$\overline{\Omega}$. Then, by Fatou's Lemma and \eqref{choice}, we have
\begin{align*}
\mu(\overline{\Omega}) = \int_{\overline{\Omega}}
\liminf_{k \to + \infty} \phi_k(x) \, d\mu
& \leq \liminf_{k \to +\infty}\int_{\overline{\Omega}}\phi_k(x) \, d\mu \\
& = \liminf_{k \to +\infty}\lim_{n \to + \infty}\left ( \int_{\Omega}
\phi_k(x) f(\chi_n(x),\nabla u_n (x)) \, dx + 
\int_{\Omega}\phi_k(x) \, d |D\chi_n| \right )\\
& \leq \lim_{n \to + \infty} \left (\int_{\Omega}
f(\chi_n(x),\nabla u_n (x)) \, dx + |D\chi_n|(\Omega) \right )
=  {\cal F}(\chi,u;\Omega),
\end{align*}
so that
\begin{equation}\label{3.5Fmy}
 \mu(\overline{\Omega})\leq {\cal F}(\chi, u; \Omega).
\end{equation}
On the other hand, by the upper semicontinuity of weak $\ast$ convergence of 
measures on compact sets, 
for every open set $V \subset \Omega$, it follows that
\begin{equation}\label{3.6FMy}
{\cal F}(\chi, u; V) \leq \liminf_{n \to +\infty} F(\chi_n, u_n; V) 
= \liminf_{n \to +\infty} \mu_n(V) 
\leq \limsup_{n \to +\infty} \mu_n(\overline{V})
\leq \mu(\overline{V}),
\end{equation}
which proves the upper bound inequality.
To prove the lower bound inequality, we start by considering an open set  
$V \subset \subset \Omega$ and $\e>0$. Then we can consider an open set 
$Z \subset \subset V$ such that
$$
\mu(V)- \mu(Z)< \e. 
$$
By \eqref{3.5Fmy}, \eqref{3.6FMy} and Lemma \ref{Lemma 3.4 FMy} we have
\begin{equation*}\label{last}
\mu(V)\leq\mu(Z)+\varepsilon 
=\mu({\overline \Omega})-  \mu( {\overline\Omega} \setminus Z) 
+ \varepsilon 
\leq {\cal F}(\chi, u; \Omega)- 
{\cal F}(\chi, u; \Omega \setminus {\overline Z})+ \varepsilon 
\leq {\cal F}(\chi, u; V)+\varepsilon.
\end{equation*}
Letting $\e \to 0^+$, we obtain 
$$
\mu(V)\leq {\cal F}(\chi, u; V),
$$
whenever $V$ is an open set such that $V \subset \subset \Omega$.
For a general open subset $V \subset \Omega$ we have
$$
\mu(V) = \sup \{\mu(O) : O \subset \subset V \}
\leq \sup \{{\cal F}(\chi, u; O) : O \subset \subset V \}
\leq {\cal F}(\chi, u; V),
$$ 
and this concludes the proof of the theorem under the coercivity assumption \eqref{coer}.

{\em Step 2.} We now remove the coercivity requirement. For each $\e >0$ we define 
$$
F_\e(\chi,u;\Omega):= F(\chi, u; \Omega) + \e \int_{\Omega}|\nabla u(x)|^p \, dx
$$
and we let ${\cal F}_\e(\chi, u; \Omega)$ denote its relaxed functional given by
\begin{eqnarray*}
\mathcal{F}_{\e}\left(  \chi,u;\Omega\right)   &  :=\inf\left\{  
\underset{n\rightarrow +\infty}{\lim\inf}
F_\e\left(\chi_{n},u_{n};\Omega\right)  :  u_{n}
\in W^{1,q}\left(\Omega;\mathbb{R}^{d}\right),
\chi_{n}
\in BV\left(\Omega;\left\{  0,1\right\}  \right), \right. \\
&  \hspace{3cm} \left.  u_{n}\rightharpoonup u\text{ in }
W^{1,p}\left(\Omega;\mathbb{R}^{d}\right),
\chi_{n}\overset{\ast}{\rightharpoonup}\chi\text{ in }
BV\left(\Omega;\left\{0,1\right\}  \right)  \right\}.
\end{eqnarray*}
By the previous step we know that there exists
a measure $\mu_\e$ which weakly represents ${\cal F}_\e$.
So, by \eqref{3.5Fmy}, we have
$$
\mu_\e(\overline\Omega)\leq {\cal F}_\e(\chi, u;\Omega)
\leq {\cal F}(\chi,u;\Omega) + \e \sup_n \|u_n\|_{W^{1,p}(\Omega;\Rb^d)} \leq C,
$$ 
where $u_n \in W^{1,q}(\Omega;\Rb^d)$ is an admissible sequence for 
$\mathcal{F}_{\e}\left(  \chi,u;\Omega\right)$.
Thus, up to a subsequence, which we do not relabel, $\mu_{\e}$ converges weakly $\ast$ to a 
finite, non-negative, Radon measure $\mu$. Given an open set $U\subset \O$, by \eqref{3.6FMy} it
follows that
$$
{\cal F}(\chi, u;U)\leq {\cal F}_\e(\chi, u;U) \leq \mu_\e({\overline U}),
$$
which, passing to the weak $\ast$ limit, yields 
$$
{\cal F}(\chi, u;U)\leq \mu({\overline U}).
$$
In order to prove the reverse inequality, let $\e'>0$, and, by definition of 
${\cal F}$, choose $u_n\in W^{1,q}(\Omega;\mathbb R^d)$ 
and $\chi_n\in BV(\Omega; \{0,1\})$, such that 
$u_n \rightharpoonup u$ in $W^{1,p}(\Omega;\mathbb R^d)$, 
$\chi_n \overset{\ast}{\rightharpoonup} \chi$ and
$$
\int_U f(\chi_n(x), \nabla u_n(x)) \, dx + |D \chi_n|(U)\leq {\cal F}(\chi,u; U)+ \e',
$$
for all $n \in \Nb$.
Then, for a sufficiently large $k$, we have
$$
\int_U \big(f(\chi_n(x), \nabla u_n(x)) + 
\e_k |\nabla u_n(x)|^p\big) \, dx + |D \chi_n|(U) 
\leq {\cal F}(\chi,u; U)+ 2\e',
$$
and hence 
$$
\mu_{\e_k}(U) \leq {\cal F}_{\e_k}(\chi,u; U) 
\leq {\cal F}(\chi,u; U)+ 2  \e'.
$$
Thus the result is proved by passing first to the weak $\ast$ limit 
as $\e_k \to 0^+$ and then to the limit as $\e' \to 0^+$.

The proof of the weak representation for $\mathcal{F}_1(\chi,u;\cdot)$ is analogous, replacing
the weak topology of $W^{1,p}(\Omega;\Rb^d)$ by the $BV(\Omega;\Rb^d)$ weak * topology for the 
convergence of the sequence $u_n$.
\end{proof}

\begin{remark}\label{measure representation}
{\rm Let $f$ be given by \eqref{density} satisfying \eqref{growth}, 
let $p,~q$
be as in (\ref{pqn}) and let 
$\chi \in BV\left(\Omega;\left\{0,1\right\}\right)$
and $u\in W^{1,p}\left(\Omega;\mathbb{R}^{d}\right)$. 
Let $\mathcal{F}$ be 
as in \eqref{introrelaxed} and $\mu$ be a Radon measure on 
$\overline{\Omega}$ which 
weakly represents $\mathcal{F} \left(\chi,u;\cdot\right)$. 
Arguing as in \cite[Lemma 3.6, Corollary 3.7 and Remark 3.8]{FMy}, 
and using Lemma \ref{Lemma 3.4 FMy}, the following facts hold. 
\begin{itemize}
\item[i)]  For every $U$ open subset of $\Omega$,  
$$ \mu(U)={\cal F}(\chi,u;U),
$$ 
provided that
\begin{equation}\label{3.7FMy}
\inf_{K}\{{\cal F}(\chi, u; U\setminus K): K \subset U, 
K\hbox{ compact}\}=0.
\end{equation}
Likewise for ${\cal F}_{\rm loc}$, ${\cal F}_1$ and ${\cal F}_{1,{\rm loc}}$.
\item[ii)] $\mu$ represents ${\cal F}$ if and only if there exists a 
Radon measure $\nu$ such that
\begin{equation}
\label{1.10Fmy}
{\cal F}(\chi, u; U)\leq \nu(U),
\end{equation}
for every open subset $U$ of $\Omega$. Likewise for 
${\cal F}_{\rm loc}$, ${\cal F}_1$ and ${\cal F}_{1,{\rm loc}}$.

In particular, if $\chi \in BV(\Omega;\{0,1\})$ and 
$u \in W^{1,q}(\Omega;\mathbb R^d)$, then
we can consider 
$$\nu(U):=\int_U Qf(\chi(x),\nabla u(x)) \, dx + |D \chi|(U),$$
where $Qf$ denotes the usual quasiconvex envelope of $f$ 
in the last variable.   
Following \cite[Corollary 4.5]{FMy} and exploiting standard results 
about quasiconvex envelopes (cf. \cite{D}, \cite[Theorem 8.4.1]{FL2}), we have that 
\begin{equation}\label{upperbdq}
{\cal F}(\chi, u; \Omega) \leq \int_\Omega Qf(\chi(x), \nabla u(x)) \, dx
+ |D \chi|(\Omega).
\end{equation}
Indeed, it suffices to fix $\chi \in BV(\Omega;\{0,1\})$ and consider 
$u_n\in W^{1,q}(\Omega;\mathbb R^d)$ such that $u_n \weak u$ in 
$W^{1,q}(\Omega;\mathbb R^d)$ and
$$
\lim_{n \to +\infty} \int_{\Omega}f(\chi(x), \nabla u_n(x)) \, dx
= \int_{\Omega}Qf(\chi(x),\nabla u(x)) \, dx.
$$
Inequality \eqref{upperbdq} also holds for ${\cal F}_1(\chi, u; \Omega)$.
\end{itemize}
}
\end{remark}

\begin{theorem}\label{measureFloc}
Let $p, q$ satisfy \eqref{pqn}, let $f$ be defined as in 
\eqref{density}, satisfying \eqref{growth}. Let 
$u \in W^{1,p}(\Omega;\mathbb R^d)$, then there exists a non-negative 
finite Radon measure $\lambda$ on $\Omega$ which strongly represents 
${\cal F}_{\rm loc}(\chi,u;\cdot)$.  The same holds in the case $p=1$,
$1 \leq q < \frac{N}{N-1}$ for the functional ${\cal F}_{1,{\rm loc}}(\chi,u;\cdot)$.

Moreover, if
	\begin{equation}
	\label{fconvex}
	f(b,\cdot) \hbox{ is convex for every }b \in \{0,1\}, 
	\end{equation}
	then, for every open subset $U \subset \Omega$, and every
	 $\chi \in BV\left(\Omega;\left\{0,1\right\}\right)$
	and $u \in W^{1,p}\left(\Omega;\mathbb{R}^{d}\right)$,
	\begin{equation}
	\nonumber
	{\cal F}_{\rm loc}(\chi, u;U)= \int_{U}f(\chi(x),\nabla u(x)) \, dx 
	+ |D \chi|(U) + \nu^s(\chi,u;U),
	\end{equation}
	where $\nu^s$ is a non-negative Radon measure singular with respect to the Lebesgue measure.

\end{theorem}
\begin{proof}
By Theorem \ref{1<p<q} we can find a Radon measure $\lambda$ in 
$\overline \Omega$ such that
$$
\lambda(U)\leq {\cal F}_{\rm loc}(\chi, u; U)\leq \lambda({\overline U}),
$$
for every open set $U \subset \Omega$. We want to prove that 
$$
\lambda(U)\geq {\cal F}_{\rm loc}(\chi, u; U).
$$
Consider an increasing sequence of open, bounded, smooth sets 
$U_h \subset \subset U$, $h \in \mathbb N$, such that 
$\overline{U}_h \subset U_{h+1}$ and $U=\cup_{i=1}^\infty U_i$.
By definition of ${\cal F}_{\rm loc}$, for $h \geq 3$, we can find two 
sequences 
$u_{h,n} \in W^{1,q}_{\rm loc}
(U_h \setminus \overline{U}_{h-2};\mathbb R^d)$ 
and $\chi_{h,n}\in BV(U_h \setminus \overline{{U}_{h-2}};\{0,1\})$ 
such that
$$
u_{h,n} \rightharpoonup u \hbox{ in }
W^{1,p}(U_h \setminus \overline{U}_{h-2};\mathbb R^d), \; \; \; 
   \chi_{h,n} \overset{\ast}{\rightharpoonup} \chi \hbox{ in }
BV(U_h \setminus \overline{U}_{h-2};\{0,1\}),
$$
\begin{equation}\nonumber
\displaystyle{\int_{(U_h \setminus \overline{U}_{h-2})}
f(\chi_{h,n}(x), \nabla u_{h,n}(x)) \, dx 
+ |D \chi_{h,n}|(U_h \setminus \overline{U}_{h-2})
\leq {\cal F}_{\rm loc}(\chi, u;(U_h \setminus \overline{U}_{h-2}))
+ 2^{-h}.}
\end{equation}
Up to the extraction of subsequences we may assume that the above 
convergences, as $n \to +\infty $, also hold a.e. in 
$U_h \setminus {\overline U}_{h-2}$, and that  
$$
\|u_{h, n}- u\|_{L^p(U_h \setminus {\overline U}_{h -2};\mathbb R^d)} 
\leq 2^{-h-n}\alpha_h^{-1},
$$
and 
$$
\|\chi_{h, n}- \chi\|_{L^1(U_h \setminus {\overline U}_{h -2};\{0,1\})} 
\leq 2^{-h-n}\alpha_h^{-1},
$$
for some $\alpha_h$ to be determined later.
By Lemma \ref{Lemma2.4FMy},  we can connect $u_{h,n}$ with 
$u_{h+1,n}$ across 
$U_h \setminus \overline{U}_{h-1}$. Hence, there exist $V^+_{h,n}$ and 
$V^-_{h+1, n}$ such that
$V^+_{h,n}\subset U_h\setminus {\overline U}_{h-2}$, 
$V^-_{h+1,n} \subset U_{h+1}\setminus {\overline U}_{h-1}$, 
$U_{h+1}\setminus \overline{U}_{h-2}= V^+_{h,n}\cup V^-_{h+1,n}$,
$$
{\cal L}^{N}(V^+_{h,n}\cap V^-_{h+1,n})\leq C_h 2^{-h-n}\alpha_h^{-1},
$$
and there exist 
$z_{h,n} \in W^{1,q}(U_{h+1}\setminus \overline{U}_{h-2};\mathbb R^d)$, 
and $\eta_{h,n} \in BV(U_{h+1}\setminus \overline{U}_{h-2};\{0,1\})$ 
such that 
$z_{h,n}=u_{h,n}$, $\eta_{h,n}=\chi_{h,n}$ in 
$(U_{h+1}\setminus \overline{U}_{h-2})\setminus V^-_{h+1,n}$
and $z_{h,n}= u_{h+1,n}$, $\eta_{h,n}= \chi_{h+1,n}$ in 
$(U_{h+1}\setminus \overline{U}_{h-2})\setminus V^+_{h,n}$.

Indeed, an argument entirely similar to the one exploited in Lemma 
\ref{Lemma 3.4 FMy} leads us to define
$$
\eta_{h,n}=\left\{
\begin{array}{ll}
\chi_{h,n} &\hbox{ in }(U_{h+1}\setminus 
{\overline U}_{h-2})\setminus V^-_{h+1,n},\\
\chi_{h+1,n} &\hbox{ in }(U_{h+1}\setminus 
{\overline U}_{h-2})\setminus V^+_{h,n}
\end{array}\right.
$$
in such a way that the transition between $\chi_{h,n}$ and 
$\chi_{h+1,n}$ occurs along a curve denoted by $\gamma_{h,n}$ satisfying 
\begin{equation}\nonumber
\int_{\gamma_{h,n}}|\chi_{h,n}(x)-\chi_{h+1,n}(x)| \, d {\cal H}^{N-1}(x) 
\leq C_h 2^{-h-n}\alpha_h^{-1}.
\end{equation}
This choice is possible since in $V_{h+1,n}^-\cap V^+_{h,n}$ both 
$\chi_{h+1,n}$ and $\chi_{h,n}$ converge strongly in $L^1$ to $\chi$ and 
a formula analogous to \eqref{413bis} holds.
 
Also
\begin{align*}
\int_{V^+_{h,n}\cap V^-_{h+1,n}}
f(\eta_{h,n}(x),\nabla z_{h,n}(x)) \, dx 
& \leq  C \int_{V^+_{h,n}\cap V^-_{h+1,n}}
(1+ |\nabla z_{h,n}(x)|^q) \, dx \\ 
& \leq C C_h 2^{-h-n}\alpha_h^{-1} + C_h \alpha_h^{-\tau q}2^{-q\tau(n+h)},
\end{align*}
where $\tau$ is as in Lemma \ref{Lemma2.4FMy}, and $C_h$ takes into account 
the dependence on $h$. Next we specify the choice of $\alpha_h$ so that 
$\alpha_h^{-\tau q}C_h \leq 1.$ 

Let $z_n\in W^{1,q}(\Omega \setminus U_1;\mathbb R^d)$ be given by 
$z_n= z_{h,n}$, in $V^+_{h,n}\cap V^-_{h+1,n}$, and $z_n= u_{h+1,n}$, in 
$(U_{h+1}\setminus U_{h-1})\setminus (V^+_{h,n}\cup V^{-}_{h+2,n})$, and let
$$
\eta_{n}:=\left\{
\begin{array}{ll}
\eta_{h,n},  &\hbox{ in } V^+_{h,n}\cap V^-_{h+1,n},\\
\chi_{h+1,n}, &\hbox{ in }U_{h+1}\setminus U_{h-1} \setminus 
(V^+_{h,n}\cup V^-_{h+2,n}). 
\end{array}\right.
$$
In this way $\eta_n \in BV(U\setminus \overline{U}_1;\{0,1\})$ and 
$\eta_n \overset{\ast}{\rightharpoonup} \chi$ in 
$BV(U\setminus \overline{U}_1;\{0,1\})$. 
Fix $k \in \mathbb N$, $k \geq 2$, then
\begin{align*}
\int_{U \setminus \overline{U}_k} & f(\eta_n(x),\nabla z_n(x)) \, dx
+ |D \eta_n|(U \setminus \overline{U}_k) \\
& \leq \sum_{h=k+1}^{+\infty}\left(\int_{U_h\setminus 
\overline{U}_{h-1}} f(\eta_n(x),\nabla z_n(x)) \, dx 
+ |D \eta_n|({U_h\setminus \overline{U}_{h-1}})\right)\\
& \leq \sum_{h=k+1}^{+\infty} \left\{
\int_{U_{h+1}\setminus \overline{U}_{h-1}}
f(\chi_{h+1,n}(x), \nabla u_{h+1,n}(x)) \, dx 
+ |D \chi_{h+1,n}|(U_{h+1}\setminus \overline{U}_{h-1}) \right.\\
& + \int_{U_h\setminus \overline{U}_{h-2}}
f(\chi_{h,n}(x),\nabla u_{h,n}(x)) \, dx 
+ |D \chi_{h,n}|(U_{h}\setminus \overline{U}_{h-2}) \\
&  \left. + \int_{V^{+}_{h,n}\cap V^-_{h+1,n}}
f(\eta_{h,n}(x),\nabla z_{h,n}(x)) \, dx + 
\int_{\gamma_{h,n}}|\chi_{h,n}(x)-\chi_{h+1,n}(x)| \, 
d {\cal H}^{N-1}(x) \right\} \\
& \leq \sum_{h=k}^{+\infty} \Big(2 {\cal F}_{\rm loc}
(\chi, u; (U_{h+1}\setminus \overline{U}_{h-1})) + 2^{-h-1}\Big) 
+ \sum_{h= k+1}^{+\infty} 2^{-q\tau(n+h)}
+ \sum_{h=k+1}^{+\infty} 2^{-(n+h)} \alpha_h^{-1}C_h \\
&\leq \sum_{h=k}^{+\infty} 2 \lambda (U_{h+2}\setminus U_{h-1})
+ 2^{-k}+ C2^{-q\tau(n+k)} + C2^{-(n+k)}  \\
&\leq 6 \lambda(U\setminus U_{k-1})+ 2^{-k} + C2^{-(q\tau+1)(n+k)}.
\end{align*}
The argument used in Lemma \ref{Lemma 3.4 FMy} ensures that
$z_n \rightharpoonup u$ in 
$W^{1,p}(U\setminus \overline{U_k};\mathbb R^d)$. Since we also have
$\eta_n \overset{\ast}{\rightharpoonup} \chi$ in 
$BV(U\setminus \overline{U}_k;\{0,1\})$, it follows that
$$
{\cal F}_{\rm loc}(\chi, u; U\setminus\overline{U}_{k})
\leq 6 \lambda(U\setminus U_{k-1}) + C 2^{-k(1+q\tau)}.
$$
Hence \eqref{3.7FMy} is verified and by Remark 
\ref{measure representation} we can conclude that
$$\displaystyle 
\lambda(U)={\cal F}_{\rm loc}(\chi, u; U).
$$

The proof of the strong representation for $\mathcal{F}_{1,{\rm loc}}(\chi,u;\cdot)$ is analogous, replacing
the weak topology of $W^{1,p}(\Omega;\Rb^d)$ by the $BV(\Omega;\Rb^d)$ weak * topology for the 
convergence of the sequence $u_n$.

Concerning the last part of the statement, for $f$ convex, we start by observing that by \eqref{fconvex}, Ioffe's Theorem 
\ref{thm2.4ABFvariant}  (see also \cite[Theorem 5.8]{AFP}) and 
the superadditivity of the liminf, the functional
$$\int_{\Omega}f(\chi(x),\nabla u(x)) \, dx + |D \chi|(\Omega)$$
is lower semicontinuous with respect to the $L^1$ strong convergence for $\chi$ and the $W^{1,p}$ weak convergence for $u$. 
Indeed, for every 
$u \in W^{1,p}(\Omega;\mathbb R^d)$ and
$\chi \in BV(\Omega;\{0,1\})$, and for every 
$u_n \in W^{1,q}_{{\rm loc}}(\Omega;\mathbb R^d)$ and 
$\chi_n\in BV(\Omega;\{0,1\})$,  
such that $u_n\wto u$ in $W^{1,p}(\Omega;\mathbb R^d)$ and 
$\chi_n \to \chi$ in $L^1(\Omega;\{0,1\})$, it follows that
\[
\begin{split}
\liminf_{n\to +\infty}\left(\int_{\Omega}f(\chi_n(x),\nabla u_n(x)) \, dx 
+ |D \chi_n|(\Omega)\right)
\geq \int_{\Omega}f(\chi(x),\nabla u(x)) \, dx 
+ |D \chi|(\Omega).
\end{split}
\] 
Thus,
\begin{equation}\label{lbd}
\int_{\Omega}f(\chi(x),\nabla u(x)) \, dx + |D \chi|(\Omega) 
\leq {\cal F}_{\rm loc}(\chi, u;\Omega)\leq {\cal F}(\chi,u;\Omega),
\end{equation}
and a similar result holds in any open subset $U$ of $\Omega$.

In order to prove the opposite inequality, let $u \in W^{1,p}(\Omega;\mathbb R^d)$ and 
$\chi \in BV(\Omega;\{0,1\})$. By definition of 
$\mathcal F_{\rm loc}$, for every open subset $U \subset \Omega$,
we have that 
\begin{equation}\label{ineq}
{\mathcal F}_{\rm loc}(\chi,u; U)\leq \liminf_{n \to + \infty}
\int_U f(\chi(x), \nabla u_n(x))\,dx + |D \chi|(U),
\end{equation}
for every sequence $u_n \in W^{1,q}_{\rm loc}(U;\mathbb R^d)$, such that $u_n \rightharpoonup u$ in $W^{1,p}(U;\mathbb R^d)$.

On the other hand, \cite[Theorem 1.1]{ABF} guarantees the existence of a measure $\nu^s(u,\chi;\cdot)$, singular with respect to the Lebesgue measure, and a sequence $\overline u_n \in W^{1,q}_{\rm loc}(U;\mathbb R^d)$ such that   $\overline u_n \rightharpoonup u $ in $W^{1,p}(U;\mathbb R^d)$ and 
$$
\limsup_{n \to + \infty}\int_{U}f(\chi(x), \nabla \overline u_n(x))\,dx \leq \int_Uf(\chi(x), \nabla u(x))\,dx + \nu^s(\chi,u; U).$$
This, together with  \eqref{ineq}, ensures that
\begin{equation}\label{ubconvex}
{\mathcal F}_{\rm loc}(\chi,u; U)\leq \int_U f(\chi(x), \nabla u(x))\, dx + |D \chi|(U) + \nu^s(\chi,u; U)
\end{equation}
and concludes the proof of the upper bound.

Hence, \eqref{lbd} applied in an open set $U \subset \Omega$, \eqref{ubconvex} and the first part of this theorem yield the result.
\end{proof}

\begin{remark}\label{noncoerciandnotstrong}
{\rm \begin{itemize}

\item[i)] We recall that, as observed by Acerbi and Dal Maso in 
\cite{ADM}, 
the exponents considered in the previous result cannot be improved, as
neither $\cal F$ nor ${\cal F}_{\rm loc}$ admit any weak representation if 
$q=\frac{Np}{N-1}$.

\item[ii)] Furthermore, it was shown in \cite[Remark 3.3]{FMy}, when there is 
no dependence on the $\chi$ variable, that, in general, the measure 
representation for ${\cal F}$ is only weak, while \cite[Theorem 3.1]{FMy} 
ensures that ${\cal F}_{\rm loc}$ admits a strong representation.

\item[iii)]For the reader's convenience we also recall that if 
$U \subset\subset V \subset \Omega$, then
$$
{\cal F}_{\rm loc}(\chi,u;U) \leq {\cal F}(\chi, u; U)
\leq {\cal F}_{\rm loc}(\chi,u; V),
$$
thus the measures $\lambda$ and $\mu$ which represent ${\cal F}_{\rm loc}$ and 
${\cal F}$, respectively, are such that $\lambda=\mu\lfloor\Omega$.
\end{itemize}
Notice that parts ii) and iii) of this remark also hold in the 
case $p=1$.}

{\rm In the convex case, we point out that, according to iii)  
		the measure 
		$\mu(\chi, u;\cdot):= f(\chi,\nabla u)\lfloor{\cal L}^N
		+ |D \chi|(\cdot) + \nu^s(\chi,u;\cdot)$ weakly represents ${\mathcal F}(\chi, u; \cdot)$.   }
\end{remark}

\begin{remark}\label{remCosciaMuccietalia}
	
	{\rm Observe that the functionals in \cite[Theorem 6.2]{CM} are related to 
		${\cal F}$ and ${\cal F}_{\rm loc}$ when $f$ is convex. Indeed, in \cite{CM} the authors consider relaxation and integral representation of integral functionals of the type $\displaystyle \int_{\Omega}h(x,\nabla u(x)) \, dx$, when $h$ satisfies
		\begin{equation}
		\label{pxgrowth}
		|\xi|^{p(x)} \leq h(x,\xi) \leq C(1 + |\xi|^{p(x)}),
		\end{equation}
		so that the integrability exponent $p(x)$ of the admissible fields depends in a continuous or regular piecewise continuous way on the location in the body (see \cite[assumptions (2.9) and (2.1) and Definitions 2.1 and 2.2]{CM}, respectively).
		
		The bulk energy density in \eqref{F}, i.e.
		$ h(x,\xi):= \chi(x)W_1(\xi)+(1-\chi(x))W_2(\xi)$, can be seen as satisfying a generalization of \eqref{pxgrowth}, allowing for a wide discontinuity in the exponent $p(x)$ and prohibiting the use of the variable exponent space $W^{1,p(x)}$.  On the other hand, the functional spaces and the convergences involved in \eqref{F}, \eqref{introrelaxed} and \eqref{introrelaxedloc}
		are not as in \cite[Section 6]{CM}. Indeed, in \eqref{introrelaxed} and \eqref{introrelaxedloc} it is required that the approximating sequences $u_n$ are more regular in the whole domain, but converge to $u$ in a weaker sense than as expected from the coercivity condition in \eqref{pxgrowth}. The convergence stated in \cite[Section 6]{CM} is  $L^1$ strong, but the coercivity condition in \eqref{pxgrowth} might provide different bounds for $\nabla u_n$ in different subsets of $\Omega$. When $\chi$ is fixed, there are many more test sequences than in \eqref{introrelaxed} or \eqref{introrelaxedloc} thus avoiding the appearance of concentration of energy as in \cite[Example 1.15]{Mucci}, which is based on the example given in \cite[page 467]{Zhikov}.
		
		We also underline that in \cite[Theorems 1.8, 1.9 and Corollary 1.10]{Mucci} a measure representation was obtained for an energy similar to $\overline{\mathcal F}$ in \cite[Section 6]{CM}, but with the approximating sequences in $C^1\cap W^{1,p(x)}$. Like in our case, under the same regularity assumptions on $h(x,\xi)$ and on $p(x)$ as in \cite{CM}, the author is able to obtain a representation in terms of a Lebesgue integral  plus a singular measure.
			}
\end{remark}

\begin{remark}\label{threerem}
	
	{\rm Let $0\leq \theta \leq 1$ and let $F$ be as in \eqref{F}.  
		It is easy to see that the representation given in the second part of Theorem \ref{measureFloc} also holds for the functional 
		${\mathcal F}_{\rm volume}$ defined by 
		\begin{align}
		\mathcal{F}_{\operatorname*{volume}}\left(\chi,u\right)   &  
		:=\inf\left\{
		\underset{n\rightarrow +\infty}{\lim\inf}F\left(  \chi_{n},u_{n};\Omega\right)
		: u_{n}  \in W_{\operatorname*{loc}}^{1,q}
		\left(\Omega;\mathbb{R}^{d}\right),
		\chi_{n}  \in 
		BV\left(\Omega;\left\{0,1\right\}  \right),  \right. \nonumber\\
		&  \left.  u_{n}\rightharpoonup u\text{ in }W^{1,p}
		\left(  \Omega;\mathbb{R}^{d}\right), 
		~\chi_{n}\overset{\ast}{\rightharpoonup}\chi\text{ in }
		BV\left(\Omega;\left\{  0,1\right\}  \right), \frac{1}{|\Omega|}\int_{\Omega}\chi_n(x) \, dx = \theta\right\}  .\nonumber
		\end{align}
		Indeed, the lower bound inequality is obtained as in the second part of the proof of Theorem \ref{measureFloc}, observing that it suffices to 
		consider sequences $\chi_n$ whose integral in $\Omega$ amounts to $\theta$.
		Regarding the upper bound inequality, the same argument as in the last part of the proof applies, taking the 
		recovery sequence $\chi_n$ identically equal to $\chi$ with 
		$\displaystyle \frac{1}{|\Omega|}\int_\Omega \chi(x) \, dx = \theta$.}
\end{remark}
 
\subsection{A sufficient condition for lower semicontinuity}

Our aim in this subsection is to present suitable assumptions on the energy densities $W_i$ in \eqref{growth} in order to guarantee lower semicontinuity of the functional $F$ in \eqref{F}. 
Also, in the spirit of \cite{K}, it is possible to consider 
functionals related to ${\mathcal F}$ and $\mathcal F_{\rm loc}$ in 
\eqref{introrelaxed} and \eqref{introrelaxedloc}, respectively, but
with the infima taken with respect to different admissible sequences.
We introduce one such functional and compare it with the previous 
ones.

An argument entirely similar to the one in \cite[Lemma 3.1]{K} allows us to prove the following.

\begin{proposition}\label{Lemma 3.1K}
	Let $F$ be as in \eqref{F} where $f$ is as in \eqref{density} with $W_i$, $i=1,2$ satisfying the lower bound in \eqref{growth}. 
	Assume also that $f(b,\cdot)$ is closed $W^{1,p}$-quasiconvex, for every $b\in \{0,1\}$.
	Then 
	$$
	F(\chi, u) \leq \liminf_{n \to + \infty}F(\chi_n, u_n),
	$$
	for every $\chi_n \in BV(\Omega;\{0,1\})$ weakly $\ast$ converging to $\chi $ in $BV(\Omega;\{0,1\})$ and every $u_n\in W^{1,p}(\Omega;\mathbb R^d)$ weakly converging to $u$ in $W^{1,p}(\Omega;\mathbb R^d)$. 
\end{proposition}
\begin{proof}
	By the lower semicontinuity of the total variation $|D\chi|$ and the superadditivity of the liminf, it suffices to 
	consider the asymptotic behaviour of 
	$$\int_{\Omega}f(\chi_n(x), \nabla u_n(x)) \, dx$$ 
	when $u_n \in W^{1,p}(\Omega;\mathbb R^d)$ and $u_n\rightharpoonup u$ in $W^{1,p}(\Omega;\mathbb R^d)$ and $\chi_n \to \chi$ in $L^1(\Omega;\{0,1\})$. 
	Let $\mu $ be the Young measure generated by $(\chi_n, \nabla u_n)$. 
	By Proposition \ref{FMYoungmeasures}, $\mu= \delta_{\chi(x)}\otimes \nu_x$, where $\nu$ is the gradient 
	Young measure generated by $\nabla u_n$.
	By \cite[Theorem 2.2]{FMAq}, and arguing as in the first part of the proof of Theorem 3.7 therein, we have
	\begin{align*}
	\liminf_{n\to +\infty}\int_{\Omega}f(\chi_n(x),\nabla u_n(x)) \, dx & \geq 
	\int_{\Omega}\int_{\mathbb R\times \mathbb R^{d\times N}}f(a,\xi) \, d(\delta_{\chi(x)}\otimes \nu_x)(a,\xi) \, dx\\
	& \geq \int_{\Omega}\int_{\mathbb R^N}f(\chi(x),\xi) \, d\nu_x(\xi)\, dx.
	\end{align*} 
	The proof will be complete provided we guarantee that
	$$
	\int_{\mathbb R^N}f(\chi(x),\xi) \, d\nu_x(\xi) \geq f(\chi(x),\nabla u(x)) \hbox{ for a.e. } x \in \Omega,
	$$ 
	and this is a consequence of \cite[Lemma 3.1]{K}.
\end{proof}

As in \cite[Corollary 1.2]{K}, one could define for $f$ as in \eqref{density} and satisfying 
\eqref{growth}, the following functional
\begin{equation}
\nonumber
\tilde{I}(\chi,u;\Omega):=\inf\left\{\liminf_{n\to +\infty}\int_{\Omega}f(\chi_n(x),\nabla u_n(x))\, dx 
+ |D \chi_n|(\Omega)\right\},
\end{equation}
where the infimum is taken over all sequences $u_n\in W^{1,p}(\Omega;\mathbb R^d)$, converging to $u$ weakly 
in $W^{1,p}(\Omega;\mathbb R^d)$, and all sequences $\chi_n\in BV(\Omega;\{0,1\})$, converging weakly $\ast$ to 
$\chi \in BV(\Omega;\{0,1\})$.

Clearly by Proposition \ref{Lemma 3.1K}, denoting by $\tilde{f}(b,\cdot)$ the closed $W^{1,p}$- quasiconvex envelope of $f(b,\cdot)$ (see Definition \ref{qcxenv}), we have that
\begin{equation}
\label{ineqfunctionals}
\int_{\Omega}\tilde{f}(\chi(x), \nabla u(x))\,dx + |D \chi|(\Omega) \leq\tilde{I}(\chi,u;\Omega)\leq {\mathcal F}_{\rm loc}(\chi,u;\Omega)\leq {\mathcal F}(\chi,u;\Omega), 
\end{equation}
when they are all finite.

Notice that, in general, we are unable to compute these functionals, in the sense of providing an explicit representation of them.   However, in the one dimensional case, since $f^{\ast \ast}(b,\cdot)=\tilde{f}(b,\cdot)$ for every $b \in \{0,1\}$, 
the four functionals in \eqref{ineqfunctionals} coincide. Indeed, by definition, any convex function is closed $W^{1,p}$-quasiconvex since Jensen's inequality holds for any probability measure.
Conversely, any closed $W^{1,p}$-quasiconvex function is $W^{1,p}$-quasiconvex in light of \cite[Corollary 3.4]{K} and \cite[Corollary 3.2]{BM}, and the latter notion is known to be equivalent to convexity in the scalar case.


\subsection{ The one dimensional case}

Theorem  \ref{measureFloc} can be improved in the one dimensional case. In fact, in this case, no singular measure appears, which does not contradict the example given in \cite[page 467]{Zhikov}, since the latter relies on the fact that functions in $W^{1,p}(\Omega;\mathbb R^d)$, with $\Omega \subset \mathbb R^2$ and $1<p<2$, are not necessarily continuous.  In fact, the next result also generalizes to the optimal design context the result of Ben Belgacem \cite[Theorem 4.1]{BB}, since we do not assume convexity in the original density $f(b,\cdot)$.

\begin{proposition}\label{lb1D}
	Let $I$ be an open interval in $\mathbb R$, let $f$ be given by \eqref{density} with $W_i$, $i=1,2$, satisfying 
	\eqref{growth} and let $p,~q$ be such that $1<p \leq q$.
	Let $\chi \in BV\left(I;\left\{0,1\right\}\right)$
	and $u \in W^{1,p}\left(I;\mathbb{R}^{d}\right)$.
	Then, 
	\begin{equation}
		\nonumber
		{\cal F}(\chi, u;I)= \int_If^{\ast\ast}(\chi(x), u'(x)) \, dx 
		+ |D \chi|(I).
	\end{equation}

\end{proposition}
\begin{proof} The result is true when $p=q$, see \cite[Theorem 8.4.1]{FL2}, so in what follows we assume that $p < q$.

	The lower bound follows as in the proof of the last part of Theorem \ref{measureFloc}. 	Indeed, since $ f(b,\cdot)\geq f^{\ast \ast}(b,\cdot)$ for every 
	$b \in \{0,1\}$, then, for every 
	$u \in W^{1,p}(I;\mathbb R^d)$ and
	$\chi \in BV(I;\{0,1\})$, 
	\begin{equation}\label{1d}
		\int_If^{\ast\ast}(\chi(x),u'(x)) \, dx + |D \chi|(I) 
		\leq {\cal F}_{\rm loc}(\chi, u;I)\leq {\cal F}(\chi,u;I).
	\end{equation}
	
	To show the upper bound, let $u \in W^{1,p}(I;\mathbb R^d)$ and 
	$\chi \in BV(I;\{0,1\})$. 
	We follow the proof in \cite[Remark 4.6]{FMy} and take a 
	convolution kernel $\rho \geq 0$, with support on $[-1,1]$ and
	such that $\int_{\mathbb R}\rho(x) \, dx=1$. Given $k \in \mathbb N$, set
	$\rho_k(x) = k\rho(kx)$ and consider the usual mollification 
	$\rho_k\ast u$. Since $u$ is continuous up to $\overline I$ we can extend it to a larger interval $J \supset \supset I$. By 
	standard relaxation results \cite[Theorem 8.4.1]{FL2}, and since $\chi$ is fixed,
	for each $k$, there exists 
	a sequence $v_{k,n} \in W^{1,q}(I;\mathbb R^d)$ such that
	$$
	v_{k,n} \wto \rho_k \ast u \hbox{ weakly in }W^{1,q}(I;\mathbb R^d )
	\hbox{ as } n \to +\infty,
	$$
	and
	$$
	\lim_{n \to +\infty} \int_I  f(\chi(x),  v'_{k,n}(x)) \, dx
	= \int_I f^{\ast \ast}(\chi(x),(\rho_k \ast u)'(x)) \, dx. 
	$$
	As $p < q$, we may extract a diagonal subsequence $u_k:= v_{k, n(k)}$ such that
	\begin{equation}\label{blabla}
	\|u_k - \rho_k \ast u\|_{W^{1,p}} \leq \frac{1}{k},
	\end{equation}
	and
	$$
	\left|\int_I f(\chi(x),  u_k'(x)) \, dx - 
	\int_I f^{\ast \ast}(\chi(x), (\rho_k \ast u)'(x)) \, dx\right|
	\leq \frac{1}{k}.
	$$
	Therefore $u_k \to u$ in ${W^{1,p}}(I;\mathbb R^d)$ 
	and
	\begin{align}\label{ubd1}
		{\cal F}(\chi,u;I)& \leq 
		\liminf_{k \to +\infty}\int_I f(\chi(x), u_k'(x)) \, dx 
		+ |D \chi|(I) \nonumber \\
		& = \liminf_{k \to +\infty} 
		\int_I f^{\ast \ast}(\chi(x),  (\rho_k \ast u)'(x)) \, dx + |D \chi|(I).
	\end{align}
	Notice that
	$f^{\ast \ast}(b, \xi)= b W_1^{\ast \ast}(\xi)+ (1-b)W_2^{\ast \ast}(\xi)$, 
	since $b$ is either $0$ or $1$. As this function is convex in the 
	second variable, we can exploit the fact that the measures $\mu^k_x$ defined by
	$$
	\left\langle\mu^k_x,\varphi\right\rangle:=\int_{\mathbb R} \rho_k(x-y)\varphi(y) \, dy
	$$
	are probability measures. Since $\chi \in BV(I;\{0,1\})$, it has finitely many discontinuity points, so we can divide the interval $I$ into a finite union of intervals $\{I_j\}_{j=1}^l$, corresponding to the largest connected sets where $\chi$ is constant. We can assume without loss of generality that each $I_j$ is open. 
	
	 Hence, using Jensen's inequality, 
	we have 
	\begin{align}
		\liminf_{k \to +\infty} 
		\int_I  f^{\ast \ast}(\chi(x), (\rho_k \ast u)'(x)) \, dx &=	\liminf_{k \to +\infty} \sum_{j=1}^l\int_{I_j}f^{\ast \ast}(\chi(x), (\rho_k \ast u)'(x)) \, dx \nonumber \\
		 = \liminf_{k\to +\infty}\sum_{j=1}^l\int_{I_j}
		W_i^{\ast \ast} ((\rho_k \ast u)'(x)) \, dx 
		 &= \liminf_{k \to +\infty} \sum_{j=1}^l\int_{I_j} 
		W_i^{\ast \ast}(\left\langle\mu^k_x, u'\right\rangle) \, dx 
	 \nonumber\\
		 \leq \limsup_{k \to +\infty}\sum_{j=1}^l\int_{I_j} 
		\left\langle\mu^k_x, W_i^{\ast \ast}(u')\right\rangle \, dx 
		& = \int_I f^{\ast \ast}(\chi(x),  u'(x)) \, dx, \nonumber
	\end{align}
	where $W^{\ast \ast}_i$ is either $W_1^{\ast \ast}$ or $W_2^{\ast \ast}$ depending on whether $\chi$ is $1$ or $0$
in each $I_j$.
	From \eqref{ubd1}, we obtain
	$$
	{\cal F}(\chi, u; I) \leq \int_I f^{\ast \ast}(\chi(x), u'(x)) \, dx 
	+|D \chi|(I),
	$$
	which, together with the lower bound obtained in the first part of the 
	proof, yields
	\begin{equation}
		\nonumber
		{\cal F}(\chi, u; I)=\int_I f^{\ast \ast}(\chi(x),u'(x)) \, dx 
		+|D \chi|(I).
	\end{equation}
\end{proof}
\begin{remark}	\label{noLavrentieff}
{\rm 	Recalling Remark \ref{threerem}, we observe that, when $\Omega$ is replaced by an interval $I$ of $\mathbb R$, 
	${\mathcal F}, \mathcal F_{\rm loc}$ and $\mathcal F_{\rm volume}$ coincide with the functional
	\begin{align}
	\widetilde{\mathcal{F}}_{\operatorname*{volume}}\left(  \chi,u\right)   &  
	:=\inf\left\{
	\underset{n\rightarrow +\infty}{\lim\inf}F\left(  \chi_{n},u_{n};I\right)
	:  u_{n}  \in W_{\operatorname*{loc}}^{1,p}
	\left(I;\mathbb{R}^{d}\right),
	\chi_{n}  \in
	BV\left(I;\left\{0,1\right\}  \right),  \right. \nonumber\\
	&  \left.  u_{n}\rightharpoonup u\text{ in }W^{1,p}
	\left(  I;\mathbb{R}^{d}\right), 
	~\chi_{n}\overset{\ast}{\rightharpoonup}\chi\text{ in }
	BV\left(I;\left\{  0,1\right\}  \right), \frac{1}{|I|}\int_{I}\chi_n(x) \, dx = \theta\right\}  .\nonumber
	\end{align}}
\end{remark}

\begin{remark}\label{weakconvW11}
{\rm We define
	\begin{align}
	\mathcal{G}_1\left(\chi,u;I\right)   &  :=\inf\left\{  
	\underset{n\rightarrow +\infty}
	{\lim\inf}\, F\left(\chi_{n},u_{n};I\right)  :
	u_{n}\in W^{1,q}\left(  I;\mathbb{R}^{d}\right),
	\chi_{n}\in BV\left(  I;\left\{  0,1\right\}  \right),\right. 
	\nonumber\\
	&  \hspace{3cm} \left.  u_{n}\rightharpoonup u
	\text{ in } W^{1,1}\left(I;\mathbb{R}^{d}\right),
	\chi_{n}\overset{\ast}{\rightharpoonup}\chi
	\text{ in } BV\left(I;\left\{0,1\right\}  \right)  \right\}, \nonumber
	\end{align}
	and, by replacing $W^{1,q}$ by $W^{1,q}_{\rm loc}$, 
	$\mathcal{G}_{1, {\rm loc}}\left(\chi,u;I\right)$. 
	Arguing as in the previous proof we can obtain a representation for both ${\mathcal G}_1$ and ${\mathcal G}_{1, \rm loc}$ in terms of $\displaystyle \int_I f^{\ast \ast}(\chi(x), u'((x)) \,dx + |D \chi|(I)$. Indeed, the lower bound inequality is a consequence of Theorem \ref{thm2.4ABFvariant}, whereas the upper bound inequality holds replacing the $W^{1,p}$ strong convergence in \eqref{blabla} by strong convergence in $W^{1,r}$ for some $1<r<q$, which in turn ensures weak convergence in $W^{1,1}(I;\mathbb R^d)$. }
	
\end{remark}

\section{Applications to strings}

In the sequel we apply the techniques of the second part of the proof of Theorem \ref{measureFloc} to identify the optimal design of strings by means of dimension reduction, in the spirit of the models described in \cite{FF, BFF}, which also appear in the context of brutal damage evolution. Namely one can deduce, as a rigorous 3$D$-1$D$ $\Gamma$-limit as $\varepsilon \to 0^+$, the optimal design of an elastic string $\Omega(\varepsilon):= B(0,\varepsilon)\times (0,l)$, with $B(0,\varepsilon) \subset \mathbb R^2$ the ball centered at $0$ with radius $\varepsilon$ and $l\in \mathbb R^+$, constituted by Ogden type 
materials, which truly exhibit a gap between the growth and coercivity exponents in the hyperelastic density. 

In the following we adopt the standard scaling (see \cite{CZ1} and the references quoted therein) which maps $x\equiv (x_1,x_2,x_3)\in \Omega(\varepsilon) \to (\frac{1}{\varepsilon} x_1,\frac{1}{\varepsilon} x_2, x_3)\in \Omega:=B(0,1)\times (0,l)$, in order to state the problem in a fixed domain (see \eqref{FDR} below). We also denote by $\nabla_\alpha u$ and $D_\alpha \chi$, respectively, the partial derivatives of $u$ and $\chi$ with respect to $x_\alpha\equiv(x_1,x_2)$.

In the model under consideration, the sequence 
$\chi_\varepsilon \in BV(\Omega;\{0,1\})$ represents the design regions, whereas $u_\varepsilon \in W^{1,q}(\Omega;\mathbb R^3)$
is the sequence of deformations, which are possibly clamped at the extremities of the string.
Standard arguments in dimension reduction (see \cite{ABP} and \cite{CZ1}) ensure that energy  bounded sequences (see the term in square brackets of \eqref{FDR}), converge (up to a subsequence), in the relevant topology, to fields $(\chi, u)$ such that $D_\alpha \chi$ and $\nabla_\alpha u$ are null, thus they can be identified, with an abuse of notation, with fields $(\chi, u)\in BV((0,l);\{0,1\}) \times W^{1,p}((0,l);\mathbb R^3)$.
In what follows we use this notation.

\begin{proposition}\label{3D1DOgden}
	
	Let $\Omega:= B(0,1)\times [0,l)$, where $B(0,1)$ denotes the unit ball in $\mathbb R^2$ and $l \in \mathbb R^+$. 
	Let $f:\{0,1\} \times \mathbb R^{3 \times 3}\to \mathbb R$ be a continuous function as in \eqref{density} and
	assume also that 
	\begin{equation}\label{Qffastast}
	Qf(b,\cdot)=f^{\ast \ast}(b,\cdot), 
	\end{equation}
	where $Q(\cdot)$ denotes the quasiconvex envelope of $f(b,\cdot)$ (see Definition \ref{qcxenv}).
	Let $1<p \leq q < +\infty$ and
	assume that there exist $c, c_0, C \in \mathbb R^+$ such that
	\begin{equation}\label{Ogdengrowth}
		c|\xi|^p- c_0\leq f(b,\xi)\leq C(1+|\xi|^q),
	\end{equation}
	for every $ b \in \{0,1\}$ and $\xi \in \mathbb R^{3\times 3}$.
	Let 

\begin{align}\label{FDR}
			\displaystyle{\mathcal F}^{DR}(\chi,u)&:=\inf \left\{ \liminf_{\varepsilon \to 0^+}\left[
			\int_{\Omega}f(\chi_\varepsilon(x), \left(\tfrac{1}{\varepsilon}\nabla_\alpha u_\varepsilon(x), \nabla_3 u_\varepsilon(x))\right) \,dx 
			+ \left|\left(\tfrac{1}{\varepsilon}D_\alpha \chi_\varepsilon, D_3 \chi_{\varepsilon}\right)\right|(\Omega)\right]: \right.\\ 	\nonumber
		\\ & \hspace{0,3cm}
\left.	u_\varepsilon\in W^{1,q}(\Omega;\mathbb R^3), 
			\chi_\varepsilon\in BV(\Omega;\{0,1\}),
			u_\varepsilon \rightharpoonup u 
            \hbox{ in }W^{1,p}(\Omega;\mathbb R^3), 
				\chi_\varepsilon \weakstar \chi \hbox{ in } BV(\Omega;\{0,1\})\right\}.
		\nonumber
	\end{align}
	Then 
	\begin{equation}\label{repFDR}
		{\mathcal F}^{DR}(\chi,u)= \pi \left(\int_{0}^l f^{\ast \ast}_0(\chi(x_3),u'(x_3)) \,dx_3 + |D \chi|(0,l)\right),
	\end{equation}
	for every $\chi \in BV((0,l);\{0,1\})$ and every $u \in W^{1,p}((0,l);\mathbb R^3)$ for which ${\mathcal F}^{DR}(\chi,u)$ 
	is finite, where
	\begin{equation}\nonumber
		f_0(b,\xi_3):=\inf_{(\xi_1,\xi_2)\in \mathbb R^{2\times 3}}f(b,\xi_1,\xi_2,\xi_3),  
		\hbox{ with } b\in \{0,1\}, (\xi_1,\xi_2,\xi_3) \equiv \xi \in \mathbb R^{3\times 3},
	\end{equation}
	and $f_0^{\ast \ast}$ denotes its convex envelope with respect to the second variable.
\end{proposition}

We point out that the functional ${\mathcal F}^{DR}$ in \eqref{FDR} is defined in full analogy with $\mathcal F$ in \eqref{relaxed1}, although it involves an asymptotic process which can be rigorously treated in the framework of $\Gamma$-convergence (we refer to \cite{DM} for more details on this subject). 
On the other hand, our proof of the integral representation \eqref{repFDR} is obtained following the same strategy, based on proving a double inequality, adopted at the end of the previous section, and it is self contained.

Notice that in the above result the limit total variation $|D\chi|$ is the counting measure. Before addressing its proof we start by proving a lemma following the ideas presented in 
\cite[Lemma 2.3]{BFMbending}.

\begin{lemma}\label{BFMbendingLemma2.3} Under the conditions of Proposition \ref{3D1DOgden} the following holds 
\begin{align}
			{\mathcal F}^{DR}(\chi,u)&:=\inf \left\{ \liminf_{\varepsilon \to 0^+}
				\left[\int_{\Omega}f^{\ast \ast}(\chi_\varepsilon(x), 
				\left(\tfrac{1}{\varepsilon}\nabla_\alpha u_\varepsilon(x), \nabla_3 u_\varepsilon(x))\right) \, dx 
				+ \left|\left(\tfrac{1}{\varepsilon}D_\alpha \chi_\varepsilon, D_3 \chi_\varepsilon\right)\right|(\Omega)\right]:  
			 \right.\\ \nonumber
			\\    \nonumber
			& \hspace{0,3cm}\left.	u_\varepsilon\in W^{1,q}(\Omega;\mathbb R^3), \chi_\varepsilon\in BV(\Omega;\{0,1\}), 
				u_\varepsilon \rightharpoonup u 
                \hbox{ in }W^{1,p}(\Omega;\mathbb R^3), 
				\chi_\varepsilon \weakstar \chi \hbox{ in } BV(\Omega;\{0,1\})\right\},\nonumber
\end{align}
	for every $(\chi, u)\in BV(\Omega;\{0,1\})\times W^{1,p}(\Omega;\mathbb R^3)$ for which
	${\mathcal F}^{DR}(\chi,u)$ is finite.
\end{lemma}
\begin{proof}
Using the arguments presented in \cite[(2.2)]{BFMbending} we obtain that
	\begin{equation}\label{2.2BFM}
		(Qf)_\varepsilon(b,\xi)= Q(f_\varepsilon)(b,\xi),
	\end{equation} 
	where for any function $g:\{0,1\}\times \mathbb R^{3\times 3}\to [0,+\infty)$, 
	$$ g_\varepsilon(b,\xi_1,\xi_2,\xi_3):=g\left(b,\tfrac{1}{\varepsilon}\xi_1,\tfrac{1}{\varepsilon}\xi_2, \xi_3\right).$$
	
	\noindent Recall that for any function $g:\{0,1\}\times \mathbb R^{3\times 3}\to \mathbb R$ the convex envelope with respect to the second variable
	\begin{equation}\nonumber
	g^{\ast \ast}(b,\xi)=\sup_{\varphi \in L^1(U;\mathbb R^{3\times 3})}\left\{\int_{U}g(b,\xi +\varphi(x)) \, dx: \int_{U}\varphi(x) \, dx=0\right\},
	\end{equation}
	where $U$ is a domain and ${\mathcal L}^3(U)=1$, coincides with the Legendre-Fenchel conjugate of $g$.
	
	By the definition of Legendre-Fenchel conjugate with respect to the second variable, it follows that
\begin{equation}\label{2.2BFMconv}
(g_\varepsilon)^\ast(b,\xi^\ast)= (g^{\ast})_{\tfrac{1}{\varepsilon}}(b,\xi^\ast),
\end{equation}
for every $b \in \{0,1\}$ and $\xi^\ast \in \mathbb R^{3\times 3}$.
	Thus, using \eqref{2.2BFMconv}, \eqref{Qffastast} and \eqref{2.2BFM} we have the following chain of equalities
	$$
	(f_\varepsilon)^{\ast\ast}(b,\xi)= (f^{\ast \ast})_\varepsilon(b,\xi) = (Qf)_\varepsilon(b,\xi)= Q(f_\varepsilon)(b,\xi),
	$$	
	for every $b \in \{0,1\}$ and $\xi \in \mathbb R^{3\times 3}$.
	
	Let ${\mathcal F}^{DR}_{f^{\ast\ast}}(\chi,u)$ be defined as ${\mathcal F}^{DR}(\chi,u)$ but replacing $f$ by $f^{\ast \ast}$.
	Clearly, since $f^{\ast\ast}\leq f$, it follows that ${\mathcal F}^{DR}_{f^{\ast\ast}}\leq {\mathcal F}^{DR}$
	so we only need to prove the opposite inequality. 
	To this end, for every $\delta>0$ and every $(\chi, u)\in BV(\Omega;\{0,1\})\times W^{1,p}(\Omega;\mathbb R^3)$ 
	for which ${\mathcal F}^{DR}(\chi,u)<+\infty$, let
	$(\chi_\varepsilon, u_\varepsilon)\in BV(\Omega;\{0,1\})\times W^{1,q}(\Omega;\mathbb R^3)$ be such that
	$u_\varepsilon \weak u$ in $W^{1,p}(\Omega;\mathbb R^3)$, $\chi_\varepsilon \weakstar \chi$ in $BV(\Omega;\{0,1\})$
	and
	$$
	{\mathcal F}^{DR}_{f^{\ast\ast}}(\chi,u)\geq \int_{\Omega}f^{\ast \ast}
	\left(\chi_\varepsilon(x),( \tfrac{1}{\varepsilon}\nabla_\alpha u_\varepsilon(x), \nabla_3 u_\varepsilon(x)) \right) \, dx 
	+\left|\left(\tfrac{1}{\varepsilon}D_\alpha \chi_\varepsilon, D_3 \chi_\varepsilon\right)\right|(\Omega)-\delta.
	$$
	By \cite[Theorem 8.4.1]{FL2}, there exists
	$ u_{\varepsilon, k}\in W^{1,q}(\Omega;\mathbb R^3)$ 
	such that 
	$
	 u_{\varepsilon, k} \rightharpoonup u_{\varepsilon} $ weakly in $W^{1,q}$, as $k \to + \infty$,
	and
	\begin{align}
		&\int_{\Omega}f^{\ast \ast}\left(\chi_\varepsilon(x),( \tfrac{1}{\varepsilon}\nabla_\alpha u_\varepsilon(x), 
		\nabla_3 u_\varepsilon(x)) \right) \, dx +\left|\left(\tfrac{1}{\varepsilon}D_\alpha \chi_\varepsilon, 
		D_3 \chi_\varepsilon\right)\right|(\Omega)\\
		&= \lim_{k\to +\infty}\int_{\Omega}f\left(\chi_\varepsilon(x), (\tfrac{1}{\varepsilon}\nabla_\alpha u_{\varepsilon,k}(x), 
		\nabla_3 u_{\varepsilon,k}(x)) \right) \, dx +\left|\left(\tfrac{1}{\varepsilon}D_\alpha \chi_\varepsilon, 
		D_3 \chi_\varepsilon\right)\right|(\Omega).\nonumber
	\end{align}
	Thus we can say that
	\begin{equation}\label{doubleenergy}
		{\mathcal F}^{DR}_{f^{\ast \ast}}(\chi, u)\geq \lim_{\varepsilon \to 0^+}\lim_{k\to +\infty} 
		\int_{\Omega}f\left(\chi_\varepsilon(x), (\tfrac{1}{\varepsilon}\nabla_\alpha u_{\varepsilon,k}(x), 
		\nabla_3 u_{\varepsilon,k}(x)) \right) \, dx +\left|\left(\tfrac{1}{\varepsilon}D_\alpha \chi_\varepsilon, 
		D_3 \chi_\varepsilon\right)\right|(\Omega) -\delta, 
	\end{equation}		
	and
	$$\lim_{\varepsilon \to 0^+}\lim_{k\to +\infty}\|u_{\varepsilon, k}-u\|_{L^p}=0.$$
	The growth from below in \eqref{Ogdengrowth}, the convexity of $|\cdot|^p$ and the fact that
	the weak topology is metrizable on bounded sets, ensure that there exist a diagonal sequence 
	$u_{\varepsilon_k,k}$  and a subsequence $\chi_{\varepsilon_k}$ such that
	$$
	(\chi_{\varepsilon_k},u_{\varepsilon_k,k}) \to (\chi,u) \hbox{ in } BV \hbox{-weak} \ast \times W^{1,p} \hbox{-weak}, \hbox{ as }
    k \to + \infty,
	$$ 
	the double limit in \eqref{doubleenergy} exists, and thus
	$$
	{\mathcal F}^{DR}_{f^{\ast\ast}}(\chi, u) \geq \lim_{k\to +\infty}\int_{\Omega}f\left(\chi_{\varepsilon_k}(x), 
	(\tfrac{1}{\varepsilon_k}\nabla_\alpha u_{\varepsilon,k}(x), \nabla_3 u_{\varepsilon_k,k}(x)) \right) \, dx +
	\left|\left(\tfrac{1}{\varepsilon_k}D_\alpha \chi_{\varepsilon_k}, D_3 \chi_{\varepsilon_k}\right)\right|(\Omega) -\delta,
	$$
	which, in turn, implies that
	$$
	{\mathcal F}^{DR}_{f^{\ast\ast}}(\chi, u) \geq {\mathcal F}^{DR}(\chi, u) -\delta.
	$$
	It suffices to let $\delta \to 0^+$ to conclude the proof.
\end{proof}

\begin{proof}
	[Proof of Proposition \ref{3D1DOgden}]
	
	The proof of \eqref{repFDR} is obtained by showing a double inequality.
	For what concerns the lower bound, it suffices to observe that $f^{\ast \ast}_0 \leq f$,
	\begin{equation}
		\nonumber
		c|\xi_3|^p- c_0 \leq f^{\ast \ast}_0(b,\xi_3)\leq C(|\xi_3|^q+1),
	\end{equation}
	for every $(b,\xi_3) \in \{0,1\}\times \mathbb R^3$ (see \cite{ABP}), and that the functional 
$$\int_\Omega f^{\ast \ast}_0(\chi(x_\alpha, x_3),\nabla_3 u (x_\alpha, x_3)) \, dx$$ 
is lower semicontinuous with respect to $BV$-weak $\ast \times W^{1,p}$ -weak convergence by Theorem \ref{thm2.4ABFvariant}. Thus, the superadditivity of the limit inf, the fact that $\left|\left(\tfrac{1}{\varepsilon}D_\alpha \chi_{\varepsilon}, D_3 \chi_\varepsilon\right)\right|(\Omega) \geq |D_3 \chi_\varepsilon|(\Omega)$ and the lower semicontinuity of the total variation, entail that
		\begin{equation}\nonumber
		\begin{split}
			{\mathcal F}^{DR}(\chi, u)&\geq \liminf_{\varepsilon \to 0^+} \int_{\Omega}f^{\ast \ast}_0(\chi_\varepsilon(x_\alpha, x_3),\nabla_3 u_\varepsilon (x_\alpha, x_3))\,dx +|D_3 \chi_\varepsilon|(\Omega)\\
			&\geq \pi \left(\int_0^l f^{\ast \ast}_0(\chi(x_3),u'(x_3))\,dx_3 + |D \chi|(0,l)\right).
		\end{split}
	\end{equation}

	In order to prove the upper bound, we use Lemma \ref{BFMbendingLemma2.3} and replace $f$ by $f^{\ast \ast}$.   We follow the proof of \cite[Proposition 3.3]{ABP}, first assuming that $u \in C^1((0,l);\mathbb R^3)$.
	
	Let $\chi \in BV((0,l);\{0,1\})$ and consider $(\varphi, \psi)\in C^1([0,l];\mathbb R^3)\times C^1([0,l]);\mathbb R^3)$.
	Define $\chi_\varepsilon(x):=\chi(x_3) $ and $w_\varepsilon(x):= u(x_3)+ \varepsilon(x_1\varphi(x_3)+x_2 \psi(x_3))$.
	Clearly $w_\varepsilon \to u$ in $W^{1,q}(\Omega;\mathbb R^3)$, and $\left(\tfrac{1}{\varepsilon}\nabla_\alpha w_\varepsilon, \nabla_3 w_\varepsilon\right)\to (\varphi(x_3),\psi(x_3),u'(x_3))$ strongly in $L^q$ (even uniformly). 
	Thus the bound from above in \eqref{Ogdengrowth} allows us to invoke the Dominated Convergence Theorem to obtain  	
	$$
	\displaystyle{\limsup_{\varepsilon \to 0^+}
		\int_{\Omega}f^{\ast \ast}(\chi(x_3),(\tfrac{1}{\varepsilon}\nabla_\alpha w_\varepsilon(x), \nabla_3 w_\varepsilon(x)))\,dx= \pi \int_0^l f^{\ast \ast}(\chi(x_3), \varphi(x_3),\psi(x_3),u'(x_3))\,dx_3.}
	$$
	The arbritrariness of $\varphi$ and $\psi$ yield 
	$$
	{\mathcal F}^{DR}(\chi, u)\leq \pi |D \chi|(0,l)+\inf_{(\varphi, \psi)\in (C^1([0,l];\mathbb R^3))^2}\pi \int_0^lf^{\ast \ast}(\chi(x_3),\varphi(x_3), \psi(x_3),u'(x_3))\,dx_3.
	$$
	The above inequality also holds for $u \in W^{1,p}((0,l);\mathbb R^3)$, by taking the infimum over pairs $(\varphi, \psi)\in (L^p((0,l);\mathbb R^3))^2$, using standard mollification results and the same arguments as in the proof of 
	Proposition \ref{lb1D}. Thus we conclude that,
for every $\chi \in BV((0,l);\{0,1\})$ and $u\in W^{1,p}((0,l);\mathbb R^3)$,
	$$
	{\mathcal F}^{DR}(\chi, u)\leq \pi |D \chi|(0,l)+\inf_{(\varphi, \psi)\in (L^p((0,l);\mathbb R^3))^2}\pi \int_0^lf^{\ast \ast}(\chi(x_3),\varphi(x_3), \psi(x_3),u'(x_3))\,dx_3.
	$$
	Observe that the continuity and the coercivity of $f^{\ast \ast}(b,\cdot)$, as in \eqref{Ogdengrowth}, entail
	\begin{equation}
		\label{foastast=fastast0}
		(f^{\ast \ast})_0(b,\xi_3) = f^{\ast \ast}_0(b,\xi_3)\hbox{ for all }b \in \{0,1\} \hbox{ and }\xi_3 \in \mathbb R^3,
	\end{equation}
	so, using the coercivity in \eqref{Ogdengrowth}, \eqref{foastast=fastast0} and the measurability criterion which provides the existence of $(\bar\varphi, \bar \psi) \in L^p((0,l);\mathbb R^3)$ such that
	$$
	(f_0)^{\ast\ast}(\chi(x_3),u'(x_3))=(f^{\ast \ast})_0(\chi(x_3),u'(x_3))=f^{\ast \ast}(\chi(x_3),u'(x_3),\bar{\varphi}(x_3),\bar{\psi}(x_3)),
	$$
	for every $(\chi, u)\in BV((0,l);\{0,1\})\times W^{1,p}((0,l);\mathbb R^3)$, 
	it follows that 
	$$
	{\mathcal F}^{DR}(\chi, u)\leq \pi \left( |D \chi|(0,l)+ \int_0^{l}f_0^{\ast \ast}(\chi(x_3),u'(x_3)) \,dx_3\right),
	$$
	which completes the proof. 
\end{proof}
Arguments similar to those used in Remark \ref{threerem} ensure that the volume constraint $\displaystyle \frac{1}{|\Omega|}\int_{\Omega}\chi(x) \, dx = \theta$ 
		can also be treated in the dimension reduction problem studied in Proposition \ref{3D1DOgden}.

\medskip

\noindent\textbf{Acknowledgements}.
The authors would like to thank CMAF-CIO at the Universidade de 
Lisboa, CIMA at the Universidade de \'Evora and Dipartimento di Ingegneria Industriale at the Universit\`a degli Studi di Salerno, where this 
research was carried out. We also gratefully acknowledge the support
of INDAM GNAMPA, Programma Professori Visitatori 2017. The research of ACB and EZ was partially supported by the
Funda\c c\~ao para a Ci\^encia e a Tecnologia through project
UID/MAT/04561/2013 and project UID/MAT/04674/2013, respectively.

\end{document}